\documentclass{article}

\usepackage{microtype}
\usepackage{graphicx}
\usepackage{subcaption}
\usepackage{booktabs} 
\usepackage{mathtools}
\mathtoolsset{showonlyrefs=true}

\usepackage{listings}
\usepackage{xcolor}

\lstdefinestyle{cppstyle}{
  language=C++,
  backgroundcolor=\color{gray!10},
  basicstyle=\ttfamily\footnotesize,
  keywordstyle=\color{blue!90}\bfseries,
  commentstyle=\color{gray!70}\itshape,
  stringstyle=\color{orange!90!black},
  numbers=left,
  numberstyle=\tiny\color{gray},
  stepnumber=1,
  numbersep=6pt,
  frame=single,
  rulecolor=\color{gray!30},
  breaklines=true,
  tabsize=2,
  showstringspaces=false,
  captionpos=b
}
\usepackage{hyperref}

\usepackage[preprint]{icml2026}

\usepackage{amsmath}
\usepackage{amssymb}
\usepackage{mathtools}
\usepackage{amsthm}

\newcommand{\bx}{\boldsymbol{x}}
\newcommand{\barg}{\bar{g}}
\usepackage[capitalize,noabbrev]{cleveref}

\theoremstyle{plain}
\newtheorem{theorem}{Theorem}[section]

\newtheorem{lemma}[theorem]{Lemma}
\newtheorem{corollary}[theorem]{Corollary}
\theoremstyle{definition}
\newtheorem{definition}[theorem]{Definition}
\newtheorem{assumption}[theorem]{Assumption}
\theoremstyle{remark}
\newtheorem{remark}[theorem]{Remark}

\usepackage[textsize=tiny]{todonotes}
\icmltitlerunning{TRSVR: An Adaptive Stochastic Trust-Region Method with Variance Reduction}

\begin{document}

\twocolumn[
  \icmltitle{TRSVR: An Adaptive Stochastic Trust-Region Method with Variance Reduction}



  \icmlsetsymbol{equal}{*}

  \begin{icmlauthorlist}
    \icmlauthor{Yuchen Fang}{equal,UCB_Math}
    \icmlauthor{Xinshou Zheng}{equal,BU}
    \icmlauthor{Javad Lavaei}{UCB_IEOR}
  \end{icmlauthorlist}

  \icmlaffiliation{UCB_Math}{Department of Mathematics, University of California, Berkeley, CA, USA}
  \icmlaffiliation{BU}{Boston University, MA, USA}
  \icmlaffiliation{UCB_IEOR}{Department of IEOR, University of California, Berkeley, CA, USA}

  \icmlcorrespondingauthor{Yuchen Fang}{yc\_fang@berkeley.edu}

  \icmlkeywords{Machine Learning, ICML}

  \vskip 0.3in
]



\printAffiliationsAndNotice{}  

\begin{abstract}
We propose a stochastic trust-region method for unconstrained nonconvex optimization that incorporates stochastic variance-reduced gradients (SVRG) to accelerate convergence. Unlike classical trust-region methods, the proposed algorithm relies solely on stochastic gradient information and does not require function value evaluations. The trust-region radius is adaptively adjusted based on a radius-control parameter and the stochastic gradient estimate. Under mild assumptions, we establish that the algorithm converges in expectation to a first-order stationary point. Moreover, the method achieves iteration and sample complexity bounds that match those of SVRG-based first-order methods, while allowing stochastic and potentially gradient-dependent second-order information. Extensive numerical experiments demonstrate that incorporating SVRG accelerates convergence, and that the use of trust-region methods and Hessian information further improves performance. We also highlight the impact of batch size and inner-loop length on efficiency, and show that the proposed method outperforms SGD and Adam on several machine learning tasks.
\end{abstract}
\section{Introduction}\label{intro}

We consider the unconstrained optimization problem
\begin{equation}\label{Problem1}
\min_{\boldsymbol{x} \in \mathbb{R}^{d}} f(\boldsymbol{x}) := \frac{1}{N} \sum_{i=1}^{N} f_i(\boldsymbol{x}),
\end{equation}
where each $f_i(\boldsymbol{x})$ denotes an individual component function and $f(\boldsymbol{x})$ represents the aggregate objective. Problems of the form \eqref{Problem1} arise ubiquitously in large-scale machine learning, where $f_i$ typically corresponds to the loss associated with a single data sample, and solving \eqref{Problem1} is commonly referred to as empirical risk minimization.

When both the sample size $N$ and the dimension $d$ are small to moderate, problem \eqref{Problem1} can be efficiently solved using deterministic optimization methods that directly access the full objective function $f$ and its gradients. However, such methods require evaluating the full gradient $\nabla f(\boldsymbol{x})$ or the objective value $f(\boldsymbol{x})$ at every iteration, incurring a computational cost that scales linearly with $N$. As a result, these approaches become prohibitively expensive when $N$ or $d$ is large, as is typical in modern machine learning applications.

To address this challenge, stochastic optimization methods have become the workhorse for large-scale learning. In these methods, only a single component function or a mini-batch of $f_i$'s is sampled at each iteration, and the gradient or objective value of $f$ is approximated using the sampled subset. This stochastic approximation significantly reduces per-iteration cost and enables scalable optimization. Among these methods, stochastic gradient descent (SGD) \citep{Robbins1951Stochastic,Bottou2018Optimization} is arguably the most widely used algorithm due to its simplicity and effectiveness. Numerous extensions of SGD have been proposed to improve empirical performance, including momentum-based methods \citep{Polyak1964Some,Liu2020Improved} and adaptive learning-rate schemes such as AdaGrad \citep{Duchi2011Adaptive} and Adam \citep{Kingma2014Adam}.

Despite their success, these methods rely primarily on first-order (gradient) information. In contrast, it is well known in deterministic optimization that incorporating second-order (Hessian) information can substantially accelerate convergence. In particular, under the Dennis--Moré condition \cite{Dennis1974characterization}, local convergence rates can improve from linear to superlinear or even quadratic \citep{Nocedal2006Numerical}. Moreover, second-order methods have demonstrated superior empirical performance in a variety of large-scale learning tasks, including language model training \citep{Liu2024Sophia} and computer vision applications \citep{Hollein20253dgs,lan20253dgs2}. Motivated by these observations, there has been growing interest in developing second-order or quasi-second-order methods for stochastic and large-scale optimization \citep{Byrd2011Use,Curtis2022Fully,Berahas2021Sequential,Fang2024Fully,Na2022adaptive}.

Most existing second-order methods fall into two broad algorithmic frameworks: line-search methods \citep[e.g.,][]{Berahas2021Sequential,Na2022adaptive} and trust-region methods \citep[e.g.,][]{Curtis2022Fully,Fang2024Fully}. Line-search methods first compute a search direction—often based on gradient or approximate Newton directions—and then determine a suitable step size to ensure sufficient descent. In contrast, trust-region methods jointly determine the search direction and step size by approximately minimizing a local model of the objective within a prescribed neighborhood, known as the trust region \citep{Conn2000Trust}. Trust-region methods are particularly appealing in nonconvex settings, as they do not require the Hessian approximation to be positive definite and often exhibit enhanced robustness and stability compared to line-search methods \citep[see, e.g.,][Chapter~4]{Nocedal2006Numerical}.

Another fundamental challenge in stochastic optimization is slow convergence caused by the high variance of stochastic gradient estimates. Such variance can lead to oscillatory behavior and large sample complexity, especially when high-accuracy solutions are required. To mitigate this issue, a rich body of work has developed variance reduction techniques that significantly improve convergence rates, including SVRG \citep{Johnson2013Accelerating}, SAG/SAGA \citep{Roux2012Stochastic,Defazio2014Saga}, SARAH \citep{Nguyen2017Sarah}, STORM \citep{Cutkosky2019Momentum}, and SPIDER \citep{Fang2018Spider}, all of which achieve improved sample complexity compared to SGD.

However, most existing variance reduction methods are developed in conjunction with SGD, where second-order information is largely unexplored. Only recently have variance reduction techniques been incorporated into stochastic line-search methods \citep{Berahas2023Accelerating}. Nevertheless, line-search frameworks typically require modifying the Hessian approximation to enforce positive definiteness, which can distort the underlying second-order information. In contrast, trust-region methods allow for the direct use of possibly indefinite curvature information. This observation naturally raises the following question:
\begin{center}
\textit{Can one design an efficient stochastic optimization algorithm that combines trust-region methods with variance reduction techniques?}
\end{center}

In this paper, we provide an affirmative answer to this question. Our main contributions are summarized as follows:
\begin{enumerate}
\item We propose a novel algorithm, \emph{Trust-Region with Stochastic Variance Reduction (TRSVR)}, which integrates the SVRG technique into a trust-region framework. The proposed method relies solely on stochastic gradient estimates and does not require function value evaluations. At each iteration, the trust-region radius is adaptively adjusted based on a radius-control parameter and the gradient estimate.

\item Under standard assumptions, we establish that TRSVR converges in expectation to a first-order stationary point. For nonconvex optimization, TRSVR achieves an $\mathcal{O}(\epsilon^{-2})$ iteration complexity and an $\mathcal{O}(N + N^{2/3}\epsilon^{-2})$ sample complexity, matching the optimal complexity attained by SVRG-based methods. Importantly, however, TRSVR operates in a more general setting: while classical SVRG analyses are restricted to first-order methods with fixed identity geometry, TRSVR allows stochastic, data-dependent Hessian matrices that may be correlated with the gradient estimate. Our analysis shows that SVRG-level complexity guarantees can be preserved despite this additional stochastic second-order structure, a regime not covered by existing variance reduction theory.

\item We evaluate the performance of TRSVR through extensive numerical experiments. The results demonstrate that (i) incorporating SVRG substantially accelerates convergence; (ii) trust-region methods and Hessian information further enhance the performance of variance reduction; (iii) the choice of batch size and inner-loop length plays a critical role in efficiency; and (iv) the proposed method outperforms SGD and Adam on several machine learning tasks.
\end{enumerate}

\textbf{Notation.}  We use $\|\cdot\|$ to denote $\ell_2$ norm for vectors and operator norm for matrices. For an integer $S>0$, we denote $[S]=\{1,2,\cdots, S\}$ and $[\tilde{S}]= \{0,1,\cdots, S-1\}$. We use $k$ to index outer loops and $s$ to index inner loops, and denote the iterate at the $k$-th outer loop and $s$-th inner loop as $\bx_{k,s}$. We use $g(\bx)$ to denote the full gradient. At $\bx_{k,s}$, we define $g_{k,s}\coloneqq g(\bx_{k,s})$. We use $\mathbb{E}_{k,s}[\cdot]$ to denote the expectation conditional on $\bx_{k,s}$.

\section{Related Work}\label{related_work}

\noindent\textbf{Stochastic second-order optimization methods.}
In deterministic optimization, it is well known that the proper incorporation of second-order information can substantially improve convergence rates \citep{Dennis1974characterization,Nocedal2006Numerical}. Motivated by the rapid growth of large-scale machine learning applications, there has been increasing interest in stochastic second-order optimization algorithms.

The existing literature primarily considers two stochastic frameworks: the \emph{fully stochastic} setup and the \emph{random model} setup. The fully stochastic setup is commonly adopted in the analysis of SGD-type methods, where gradient estimators are assumed to be unbiased with bounded variance. In contrast, the random model setup requires gradient and function value estimates to satisfy adaptive accuracy conditions with a fixed probability; in this setting, gradient estimators may be biased and exhibit unbounded variance. For algorithms under the fully stochastic setup, we refer the reader to \citep{Berahas2021Sequential,Fang2024Fully,Curtis2022Fully,Na2025Statistical,Kuang2025Online}. Representative results under the random model setup can be found in \citep{Blanchet2019Convergence,Fang2024Trust,Fang2025High,Na2022adaptive,Na2023Inequality}. In this work, we adopt the fully stochastic setup, as algorithms developed in this framework typically enjoy more favorable sample complexity guarantees.

\noindent\textbf{Variance reduction techniques.}
A broad class of variance reduction techniques has been developed to reduce the sample complexity of stochastic optimization methods. For $\mu$-strongly convex functions with $L$-Lipschitz continuous gradients, vanilla SGD requires $\mathcal{O}(1/\epsilon)$ samples to achieve an $\epsilon$-optimal solution, whereas variance-reduced methods such as SVRG, SAG/SAGA, and SARAH achieve linear convergence with sample complexity $\mathcal{O}((N + \kappa)\log(1/\epsilon))$, where $\kappa = L/\mu$ denotes the condition number. For general convex objectives, SGD exhibits a sample complexity of $\mathcal{O}(1/\epsilon^2)$, while SVRG achieves $\mathcal{O}(N + \sqrt{N}/\epsilon)$, SAGA attains $\mathcal{O}(N + N/\epsilon)$, and SARAH achieves $\mathcal{O}((N + 1/\epsilon)\log(1/\epsilon))$. In the nonconvex setting, SVRG attains a sample complexity of $\mathcal{O}(N + N^{2/3}/\epsilon)$. In the online setting, STORM achieves a sample complexity of $\mathcal{O}(\epsilon^{-3/2})$, while SPIDER unifies several variance reduction schemes and attains $\mathcal{O}(\epsilon^{-3/2})$ complexity in the online setting and $\mathcal{O}(N + N^{1/2}\epsilon^{-1})$ in the finite-sum setting. In this work, we focus on SVRG due to its widespread adoption and suitability for nonconvex finite-sum optimization.
\section{Algorithm Design}\label{algo_design}

In this section, we describe the design of TRSVR, summarized in Algorithm~\ref{algo_TR-SVR}. The method follows the nested-loop structure typical of SVRG-type algorithms. Specifically, at each outer iteration, the full gradient is computed using all samples, while at each inner iteration, a mini-batch of samples is used to form a stochastic gradient estimate.

We denote the iterate at the beginning of the $k$-th outer iteration by $\bx_{k,0}$ and refer to it as the \emph{reference iterate}. The full gradient at the reference iterate is computed as
\vspace{-0.15cm}
\begin{equation}\label{full_gradient}
    g_{k,0} = \frac{1}{N} \sum_{i=1}^{N} \nabla f_i(\bx_{k,0}),
\end{equation}
which is evaluated only once per outer iteration and serves as a control variate for variance reduction.

The algorithm then enters the inner loop of length $S$. At each inner iteration $s$, a mini-batch $I_{k,s} \subset [N]$ of size $b$ is sampled, and the corresponding mini-batch gradient estimator at $\bx_{k,s}$ is computed as
\vspace{-0.1cm}
\begin{equation}\label{mini_batch_gradient}
    \tilde{g}_{k,s} = \frac{1}{b} \sum_{i \in I_{k,s}} \nabla f_i(\bx_{k,s}).
\end{equation}
Using the SVRG estimator, we then form the variance-reduced gradient
\begin{align}\label{var_reduced_grad}
\barg_{k,s}
&= \frac{1}{b} \sum_{i \in I_{k,s}} \Big( \nabla f_i(\bx_{k,s}) - \nabla f_i(\bx_{k,0}) \Big) + g_{k,0} \notag \\
&= \tilde{g}_{k,s} - \big( \tilde{g}_{k,0} - g_{k,0} \big),
\end{align}
which is an unbiased estimator of the full gradient at $\bx_{k,s}$ with reduced variance.

Based on the variance-reduced gradient, we define the trust-region radius as
\begin{equation}\label{trust-region radius}
    \Delta_{k,s} = \alpha \|\barg_{k,s}\|,
\end{equation}
where $\alpha > 0$ is a user-specified radius-control parameter.

\begin{algorithm}[t]
\caption{Trust-Region Method with Stochastic Variance Reduction (TRSVR)}
\label{algo_TR-SVR}
\begin{algorithmic}[1]
\STATE \textbf{Input:} Initialization $\bx_{0,0}$, batch size $b \in [1, N]$, inner-loop length $S > 0$, radius-control parameter $\alpha > 0$.
        
\FOR{$k = 0, 1, 2, \dots$} 
\STATE Set $\bx_{k,0} = \bx_{k-1,S}$ (if $k > 0$) or initialize $\bx_{k,0}$.
\STATE Compute full gradient $g_{k,0}$ following \eqref{full_gradient}.
\FOR{$s = 0$ to $S-1$}
\STATE Select a mini-batch $I_{k,s}$ and compute the mini-batch gradient estimate $\tilde{g}_{k,s}$ following \eqref{mini_batch_gradient}.
                
\STATE Compute variance-reduced gradient $\barg_{k,s}$ as in \eqref{var_reduced_grad}.
\STATE Generate the trust-region radius $\Delta_{k,s}$ as in \eqref{trust-region radius}.
\STATE Construct $H_{k,s}$ and solve \eqref{trust-region subproblem}
to obtain $\Delta \bx_{k,s}$.
                
\STATE Update the iterate as $\bx_{k,s+1} = \bx_{k,s} + \Delta \bx_{k,s}$.
\ENDFOR
            
\STATE Update the iterate as $\bx_{k+1,0} = \bx_{k,S}$.
\ENDFOR
\end{algorithmic}
\end{algorithm}

At each inner iteration, we construct a Hessian approximation $H_{k,s}$ and compute the step $\Delta \bx_{k,s}$ by solving the trust-region subproblem
\begin{equation}\label{trust-region subproblem}
  \begin{aligned}
    \min_{\Delta \bx \in \mathbb{R}^d} \quad
    m(\Delta \bx)
    &= \barg_{k,s}^T \Delta \bx
      + \frac{1}{2} \Delta \bx^T H_{k,s} \Delta \bx \\
    \text{s.t.} \quad
    & \|\Delta \bx\| \leq \Delta_{k,s}.
  \end{aligned}
\end{equation}
\begin{remark}
We allow the Hessian approximation $H_{k,s}$ to be stochastic and potentially dependent on the gradient estimator $\barg_{k,s}$. This setting goes beyond existing variance reduction analyses, which are largely restricted to first-order methods and effectively assume $H_{k,s}=I$. Recent work incorporating variance reduction into line-search methods \citep{Berahas2023Accelerating} still requires $H_{k,s}$ to be conditionally independent of $\barg_{k,s}$.
\end{remark}

We do not require \eqref{trust-region subproblem} to be solved exactly. Instead, we assume that the computed step $\Delta \bx_{k,s}$ achieves at least as much reduction in the model $m(\cdot)$ as the Cauchy point \citep[see][]{Nocedal2006Numerical}. Under this condition, the following standard result holds (see, e.g., \citealp{Curtis2022Fully,Fang2024Fully}).
\begin{lemma}[Cauchy Reduction]\label{lemma:Cauchy reduction}
Let $\Delta \bx_{k,s}$ be an approximate solution to \eqref{trust-region subproblem} that achieves at least the Cauchy decrease. Then, for all $k \ge 0$ and $s \in [\tilde{S}]$, we have
\vspace{-0.1cm}
\begin{align*}
m(\Delta \bx_{k,s}) - m(\boldsymbol{0})
&= \barg_{k,s}^T \Delta \bx_{k,s}
+ \frac{1}{2} \Delta \bx_{k,s}^T H_{k,s} \Delta \bx_{k,s} \\
&\le - \|\barg_{k,s}\| \Delta_{k,s}
+ \frac{1}{2} \|H_{k,s}\| \Delta_{k,s}^2.
\end{align*}
\end{lemma}

\section{Algorithm Analysis}\label{algo_analysis}

In this section, we establish the theoretical properties of the proposed algorithm. We begin by stating the assumptions.

\begin{assumption}\label{assump1}
All iterates $\bx_{k,s}$ lie in an open convex set $\mathcal{X} \subseteq \mathbb{R}^d$. The objective $f: \mathbb{R}^d \to \mathbb{R}$ is bounded below by $f_{\inf}$. For each $i \in [N]$, the component function $f_i : \mathbb{R}^d \to \mathbb{R}$ is continuously differentiable with $L$-Lipschitz continuous gradient. For all $k \geq 0$ and $s \in [\tilde{S}]$, the Hessian approximation $H_{k,s}$ satisfies $\|H_{k,s}\| \leq \kappa_H$ for some constant $\kappa_H$.
\end{assumption}

\begin{assumption}\label{assump2}
The mini-batch gradient estimator $\tilde{g}_{k,s}$ is unbiased, i.e.,
$\mathbb{E}[\tilde{g}_{k,s} \mid \bx_{k,s}] = g_{k,s}$.
\end{assumption}

\begin{remark}
Assumption \ref{assump1} implies that the aggregate objective function $f$ is continuously differentiable with $L$-Lipschitz continuous gradient. The condition in Assumption \ref{assump2} is satisfied, for example, when each index in the mini-batch $I_{k,s} \subset [N]$ is sampled uniformly at random.
\end{remark}

Next, we define the first-order $\epsilon$-stationary point based on the expected squared norm of the full gradient.
\begin{definition}
   A point $\bx$ is said to be a first-order $\epsilon$-stationary point if the full gradient $g(\bx)$ satisfies 
   \begin{equation*}
       \mathbb{E}[\|g(\bx)\|^2]\leq \epsilon.   
    \end{equation*}
\end{definition}

\subsection{Main theorem}
With the above assumptions, we establish the main theorem.
\begin{theorem}[Global Convergence]\label{thm:Global Convergence_main}
Under Assumptions \ref{assump1} and \ref{assump2}, we select
$\alpha = \frac{\mu_0 b}{2(L+2\kappa_H)N^\gamma}\in(0,1]$,  $b =\mu_1  N^{\gamma}$ with $\mu_0,\mu_1,\gamma\in(0,1]$, and $S \leq \bigg\lfloor\frac{N^{3\gamma/2}}{\mu_0(b+\frac{\mu_0L^2b}{2(L+2\kappa_H)})}\bigg\rfloor$. Then for all $k\geq 0$ and $s\in [\tilde{S}]$, there exist constants $\mu_0$ and $v_0 \in (0, 1)$ such that
\begin{align*}
    & \mathbb{E}\left[\frac{1}{(K+1)S}\sum_{k=0}^K\sum_{s=0}^{S-1}\|g_{k,s}\|^2\right] \\
    & \qquad\qquad\qquad \leq \frac{2(L+\kappa_H) (\mathbb{E}[f(\bx_{0,0})] - f_{\inf})}{(K+1)S\mu_0 \mu_1  v_0}.
\end{align*}

\end{theorem}
This result implies several important properties of TRSVR. First, letting $K \to \infty$, the expected average squared gradient norm converges to zero, which shows that a subsequence of iterates generated by TRSVR converges in expectation to a first-order stationary point. Second, TRSVR enjoys an $\mathcal{O}(1/\epsilon)$ iteration complexity for finding a first-order $\epsilon$-stationary point. The corresponding sample complexity is stated below.

\begin{corollary}
Under the assumptions of Theorem \ref{thm:Global Convergence_main} with $\gamma = 2/3$, TRSVR achieves a sample complexity of $\mathcal{O}(N + N^{2/3}\epsilon^{-1})$.
\end{corollary}

The above result can be derived using an analysis similar to \citep[][Corollary~2]{Reddi2016Stochastic}, which matches the sample complexity of SVRG combined with SGD. Importantly, our analysis allows the Hessian approximation to be stochastic and potentially dependent on the gradient estimator. This demonstrates that SVRG-level sample complexity can be preserved even in the presence of stochastic second-order geometry, a regime not covered by existing variance reduction results.

\subsection{Analysis outline}

In this subsection, we outline the analysis of TRSVR, and defer the detailed analysis to Appendix \ref{append_B}.

Due to the nested-loop structure of the algorithm, our analysis first considers the reduction in objective function for one inner iteration in expectation.
\begin{lemma}[One-Step Expected Reduction]\label{lem:exp_decrease}
Under Assumptions \ref{assump1} and \ref{assump2}, suppose that $\alpha \leq \frac{1}{2(L + 2\kappa_H)}$. Then for all $k\geq 0, s\in [\tilde{S}]$, we have
\begin{align*}
     & \mathbb{E}_{k,s}[f(\bx_{k,s+1})] - f(\bx_{k,s}) \\
     & \leq -\frac{1}{4}\alpha\|g_{k,s}\|^2 + \frac{1}{2}(L + 2\kappa_H)\alpha^2\mathbb{E}_{k,s}[\|\barg_{k,s} - g_{k,s}\|^2].
\end{align*}
\end{lemma}
Then we focus on the variance term $\mathbb{E}_{k,s}[\|\barg_{k,s} - g_{k,s}\|^2]$. We demonstrate that SVRG enables us to upper-bound it by the squared distance between the current iterate (i.e., $\bx_{k,s}$) and the reference iterate (i.e., $\bx_{k,0}$). 
\begin{lemma}[Upper Bound of Variance]\label{lem:Variance Bound}
Under Assumptions \ref{assump1} and \ref{assump2}, for all $k\geq 0,s \in [\tilde{S}]$, we have
\begin{equation*}
    \mathbb{E}_{k,s}[\|\barg_{k,s} - g_{k,s}\|^2] \leq \frac{L^2}{b}\|\bx_{k,s} - \bx_{k,0}\|^2.
\end{equation*}
\end{lemma}

With the above two key lemmas, we now describe a sketch of the proof of the main theorem.

\noindent\textit{Sketch of proof of Theorem \ref{thm:Global Convergence_main}.} The proof combines the conclusions in Lemma \ref{lem:exp_decrease} and Lemma \ref{lem:Variance Bound} and the Lyapunov function 
\begin{equation*}
   \Phi_{k,s} = f(\bx_{k,s}) + \lambda_s\|\bx_{k,s} - \bx_{k,0}\|^2, 
\end{equation*}
where $\lambda_s$ has a recursive definition as
\begin{align*}
    \lambda_s = \alpha^2\frac{L^2(L+ 2\kappa_H)}{2b} + \lambda_{s+1}\left(1+\alpha z+\left(\alpha^2+\frac{\alpha}{z}\right)\frac{L^2}{b}\right)
\end{align*}
with $\lambda_S=0$. With the help of the Lyapunov function, we can show that  
\begin{equation*}
     \mathbb{E}\left[\frac{1}{(K+1)S}\sum_{k=0}^K\sum_{s=0}^{S-1}\|g_{k,s}\|^2\right] \leq \frac{\mathbb{E}[f(\bx_{0,0})] - f_{\inf}}{(K+1)S \cdot \Lambda_{\min}}.
\end{equation*}
Then, the parameter setting in Theorem \ref{thm:Global Convergence_main} enables $ \Lambda_{\min} \geq \frac{\mu_0\mu_1 v_0}{2(L+\kappa_H)}$, which completes the proof.
\section{Experiments}

We present empirical results to evaluate the performance of TRSVR. The primary goal of our experiments is to demonstrate that TRSVR converges more rapidly and stably to high-precision solutions across a range of convex and nonconvex problems. Specifically, we conduct experiments along the following four dimensions:

\textbf{(a)} Comparison with variance-reduced first-order methods, including SVRG \citep{Johnson2013Accelerating}, SAGA \citep{Defazio2014Saga}, and SARAH \citep{Nguyen2017Sarah}, to assess the benefits of incorporating second-order information.

\textbf{(b)} Comparison with first-order methods without variance reduction, such as SGD \cite{Robbins1951Stochastic} and Adam \cite{Kingma2014Adam}, to evaluate the combined benefits of variance reduction and curvature information.

\textbf{(c)} Comparison with deterministic \citep[][Algorithm~4.1]{Nocedal2006Numerical} and stochastic trust-region methods  \citep{,Curtis2022Fully} to assess the effect of variance reduction within trust-region frameworks.

\textbf{(d)} Sensitivity analysis with respect to the mini-batch size $b$ and inner-loop length $S$ to examine their trade-off under a fixed computational budget.

For all experiments, the trust-region subproblem is solved using the Steihaug conjugate gradient method \citep[][Algorithm~7.2]{Nocedal2006Numerical}. Implementation details and full results are provided in Appendix~\ref{append_C}.

\subsection{Hessian approximations}
The algorithm design and analysis allow for general constructions of the Hessian approximation $H_{k,s}$. To investigate their empirical impact, we consider two representative choices with matrix-free implementations.

\textbf{(a) Identity (Id).} We set $H_{k,s} = I$, a choice commonly adopted in stochastic  line-search methods \citep{Berahas2021Sequential,Berahas2023Stochastic,Na2022adaptive,Na2023Inequality}.

\textbf{(b) Estimated Hessian (EstH).} We construct $H_{k,s}$ as an estimate of $\nabla^2 f(\bx_{k,s})$. In practice, Hessian--vector products are computed using finite-difference approximations based on stochastic gradients, allowing the use of curvature information without incurring $\mathcal{O}(d^2)$ memory costs.

\begin{figure}[t]
    \centering
    \begin{minipage}{0.45\textwidth}
        \centering
        \includegraphics[width=\linewidth]{ 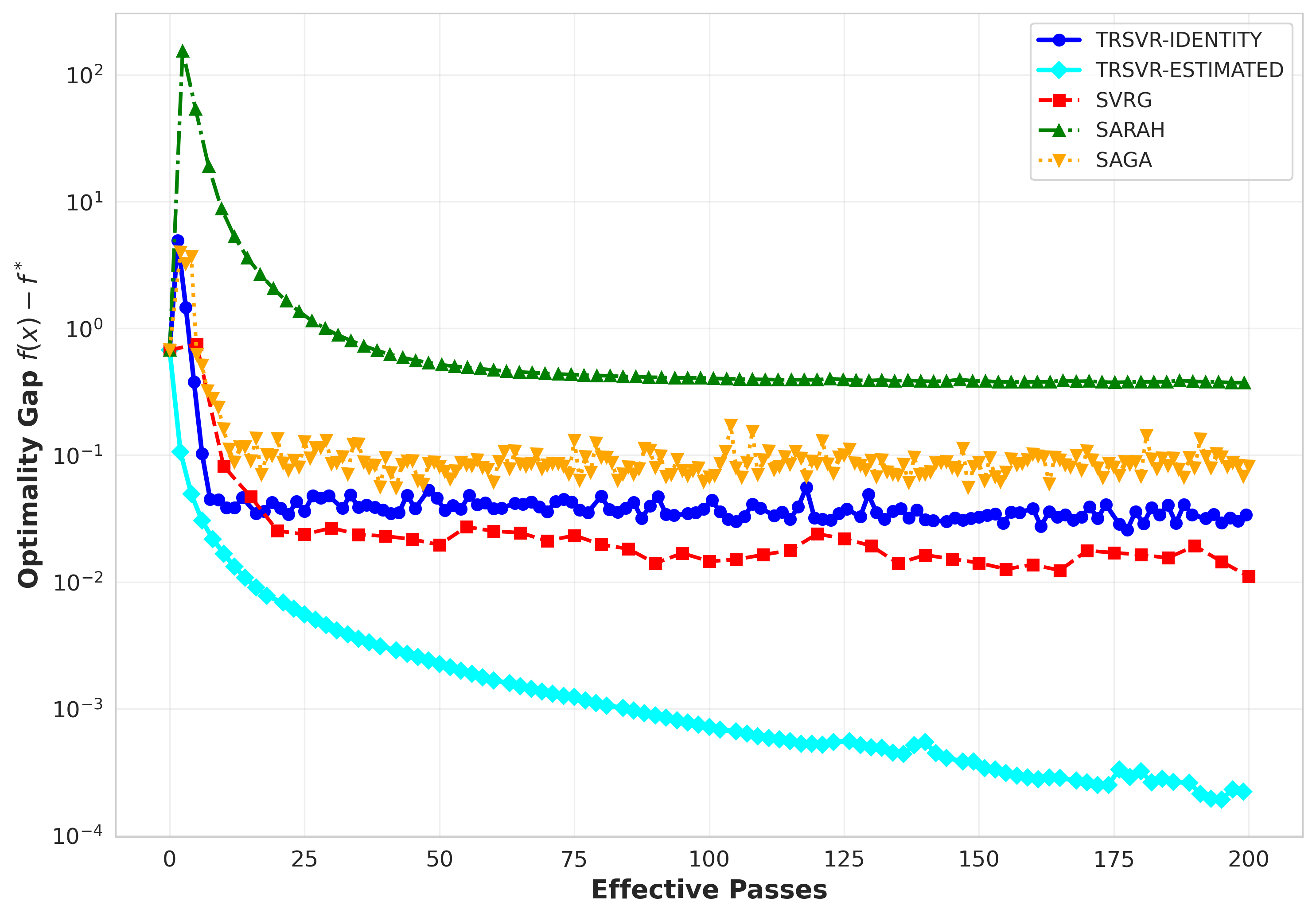}
        \caption{Optimality Gap with effective passes. Each trajectory represents a method.}
        \label{fig:obj_gap}
    \end{minipage}\hfill
    \begin{minipage}{0.45\textwidth}
        \centering
        \includegraphics[width=\linewidth]{ 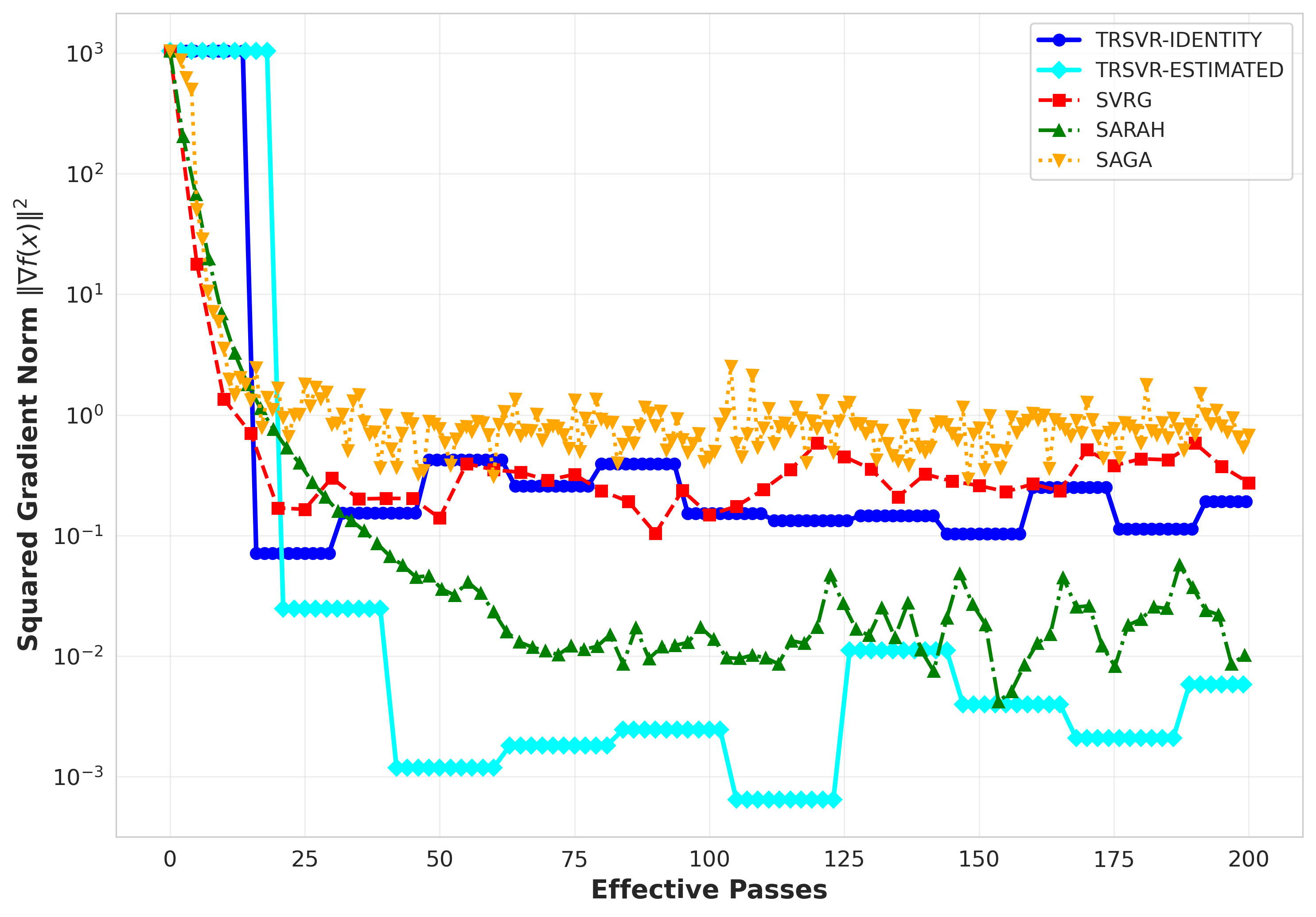}
        \caption{Squared Gradient Norm with effective passes. Each trajectory represents a method.}
        \label{fig:grad_norm}
    \end{minipage}
\end{figure}
\subsection{Comparison with SAGA, SVRG and SARAH}
We first compare TRSVR with variance-reduced first-order methods—SAGA, SVRG, and SARAH—where variance reduction is combined with SGD.

\subsubsection{Convex optimization: ill-conditioned logistic regression}
\label{sec:convex}

We consider an $\ell_2$-regularized logistic regression problem designed to be strongly convex but ill-conditioned:
\begin{equation*}
    f(\boldsymbol{w}) = \frac{1}{N}\sum_{i=1}^{N} \log(1 + \exp(-y_i \bx_i^T \boldsymbol{w})) + \frac{\lambda}{2}\|\boldsymbol{w}\|^2,
\end{equation*}
where $\lambda = 10^{-4}$. Synthetic data are generated with $N = 8\times 10^4$ samples and dimension $d = 32$, following a zero-mean Gaussian distribution. To induce ill-conditioning, the covariance matrix is constructed with condition number $\kappa \approx 10^4$, reflecting challenging optimization landscapes commonly encountered in practice.

Convergence is measured by effective passes (gradient evaluations normalized by 
$N$). Figures~\ref{fig:obj_gap} and~\ref{fig:grad_norm} show TRSVR-EstH consistently outperforms all baselines, while TRSVR-Id matches variance-reduced SGD methods. Since first-order methods converge slowly under large condition numbers, TRSVR-EstH's superior performance demonstrates that curvature information mitigates ill-conditioning, while TRSVR-Id's comparable performance confirms gains arise from second-order geometry rather than the trust-region framework alone.

\subsubsection{Nonconvex optimization}
\label{sec:nonconvex}

We next evaluate TRSVR on nonconvex problems using the \texttt{RCV1} and \texttt{Mushroom} datasets. To induce nonconvexity, we augment logistic regression with a double-well regularizer:
\begin{align}\label{obj:nonconvex}
    f(\boldsymbol{w}) & = \underbrace{\frac{1}{N}\sum_{i=1}^N \log(1+\exp(-y_i \bx_i^T \boldsymbol{w})) + \frac{\lambda}{2}\|\boldsymbol{w}\|^2}_{\text{Convex Part}} \notag \\
    & \qquad\qquad\qquad \qquad+ \underbrace{  \frac{\gamma}{d} \sum_{j=1}^d (w_j^2 - a^2)^2}_{\text{Nonconvex Regularization}},
\end{align}
with $\lambda = 10^{-4}$, $\gamma = 10^{-4}$, and $a = 0.5$. Sparse implementations of SAGA, SVRG, and SARAH are used where applicable.
We report the squared gradient norm with epochs in Figures \ref{fig:rcv1_epochs} and \ref{fig:mushroom_epochs}.

\begin{figure}[t]
    \centering
    \begin{minipage}{0.48\textwidth}
        \centering
\includegraphics[width=\linewidth]{ 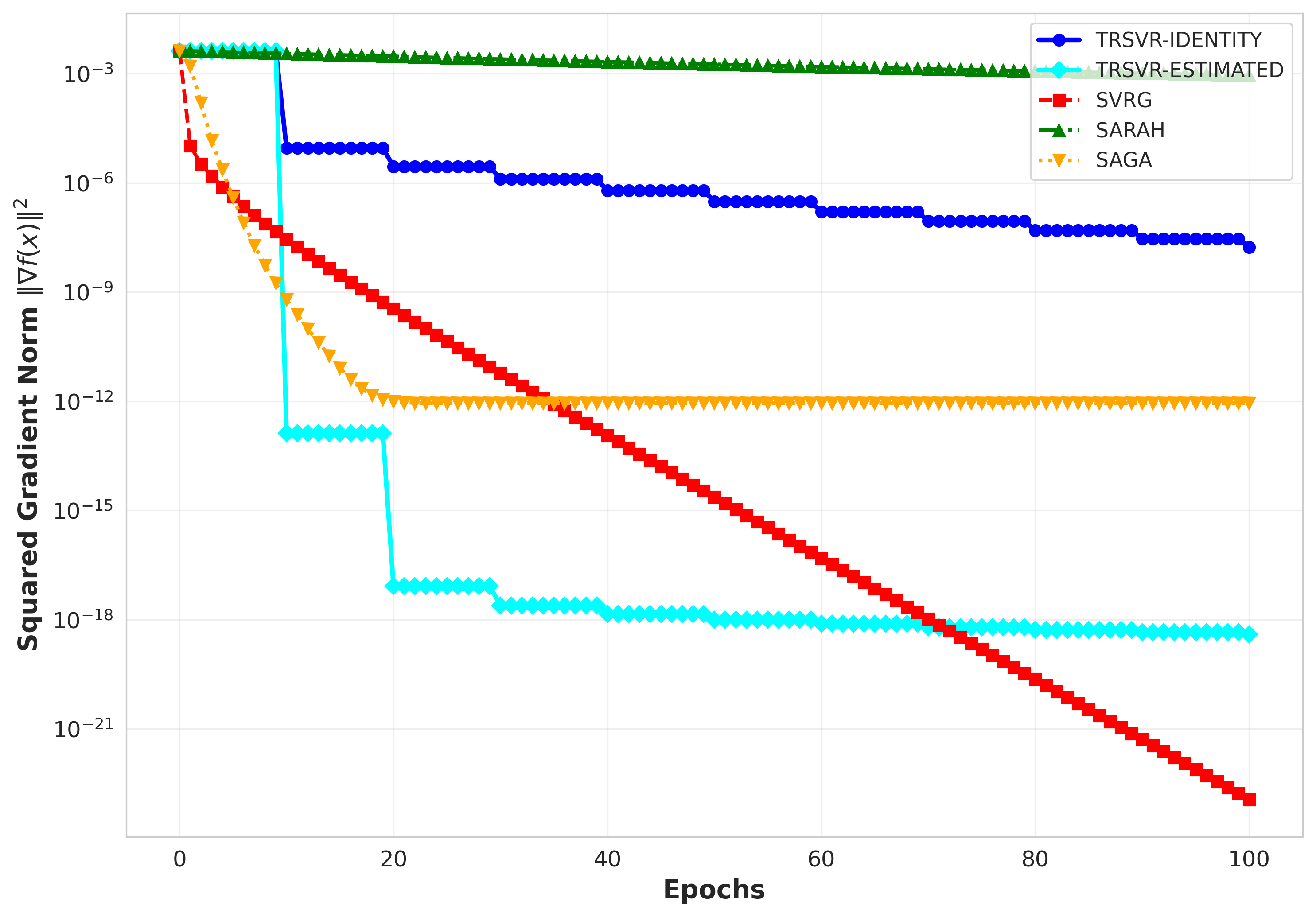}
\caption{Squared gradient norm with epochs on the \texttt{RCV1} dataset. Each trajectory represents a method.}
\label{fig:rcv1_epochs}
\end{minipage}\hfill
\begin{minipage}{0.48\textwidth}
        \centering
\includegraphics[width=\linewidth]{ 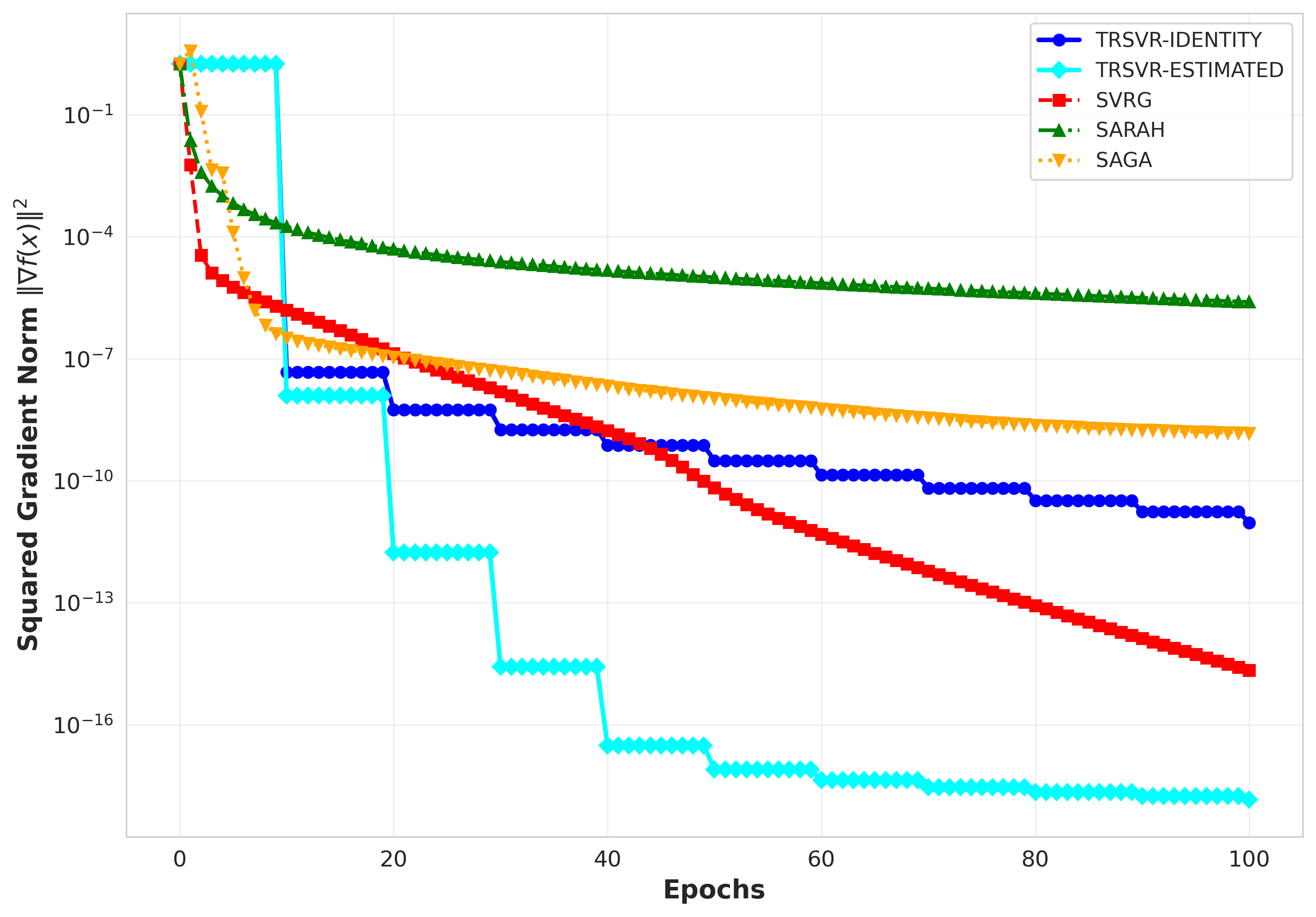}
\caption{Squared gradient norm with epochs on the \texttt{Mushroom} dataset. Each trajectory represents a method.}
\label{fig:mushroom_epochs}
\end{minipage}
\end{figure}
On both datasets, TRSVR-EstH converges significantly faster than TRSVR-Id and all variance-reduced first-order baselines, confirming that curvature information improves navigation of nonconvex landscapes. While TRSVR-EstH may converge to slightly lower final precision than SVRG in some cases, it reaches practically sufficient accuracy significantly faster, which is often the dominant consideration in large-scale learning.

\subsection{Comparison with SGD and Adam}
\label{sec:benchmark}

We further compare TRSVR-EstH with SGD and Adam, two widely used optimizers in large-scale machine learning. SGD uses momentum $0.9$, while Adam uses $\beta_1=0.9$ and $\beta_2=0.999$. Experiments are conducted on the \texttt{Covertype} and \texttt{IJCNN1} datasets using the nonconvex objective \eqref{obj:nonconvex}.

Figures~\ref{fig:Covertype_SGD_Adam} and~\ref{fig:Ijcnn1_SGD_Adam} show the optimality gap with wall-clock time. TRSVR consistently reaches high-precision solutions substantially faster, whereas SGD and Adam plateau at suboptimal accuracy levels around $10^{-2}$ to $10^{-5}$. We note that SGD and Adam benefit from highly optimized implementations and momentum acceleration, while our TRSVR implementation is not yet fully optimized, suggesting further potential performance gains.

\begin{figure}[t]
\centering
\includegraphics[width=0.9\linewidth]{ 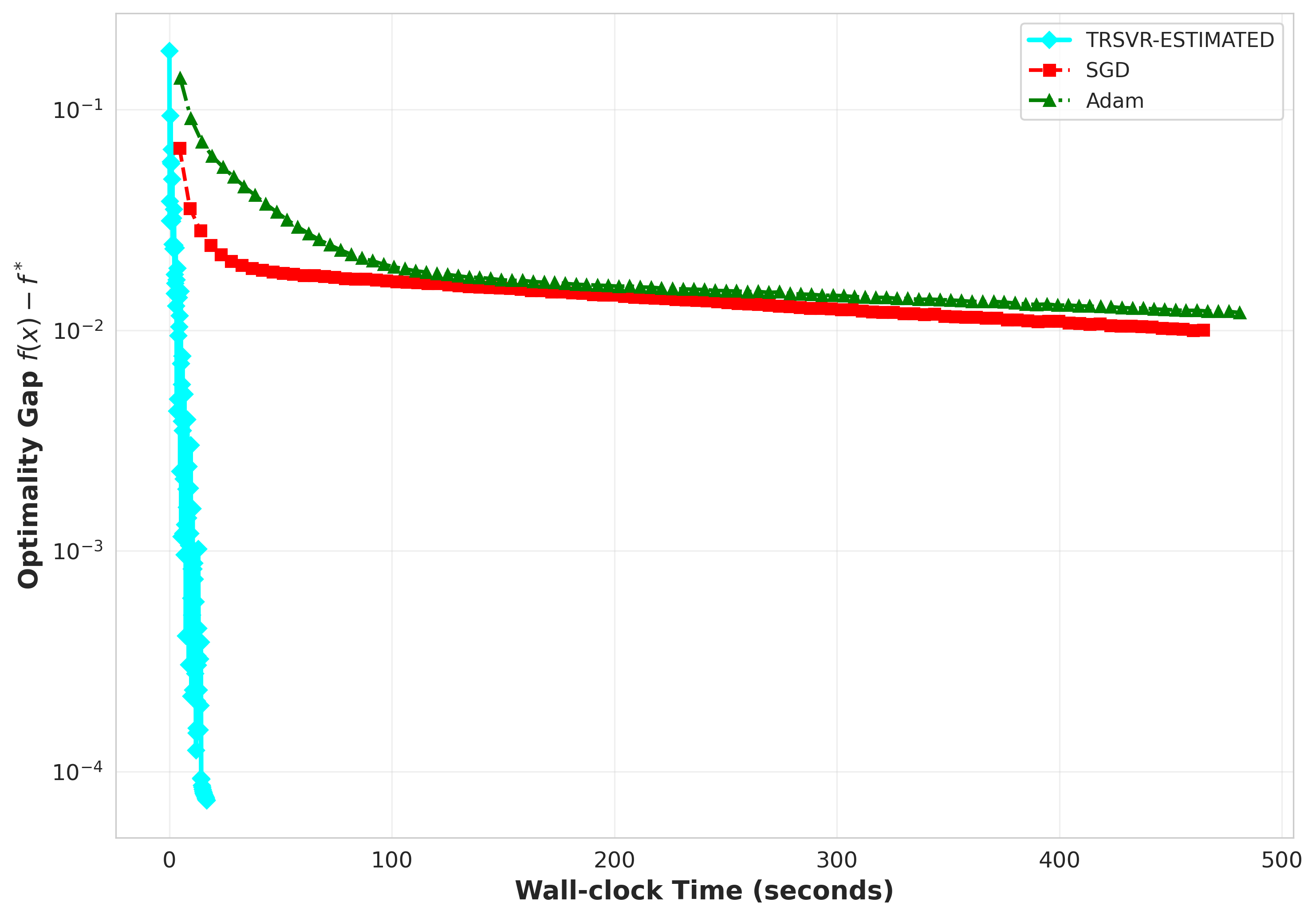}
\caption{Optimality gap with wall-clock time on the \texttt{Covertype} dataset. Each trajectory represents a method.}
\label{fig:Covertype_SGD_Adam}
\end{figure}
\begin{figure}[t]
\centering
\includegraphics[width=0.9\linewidth]{ 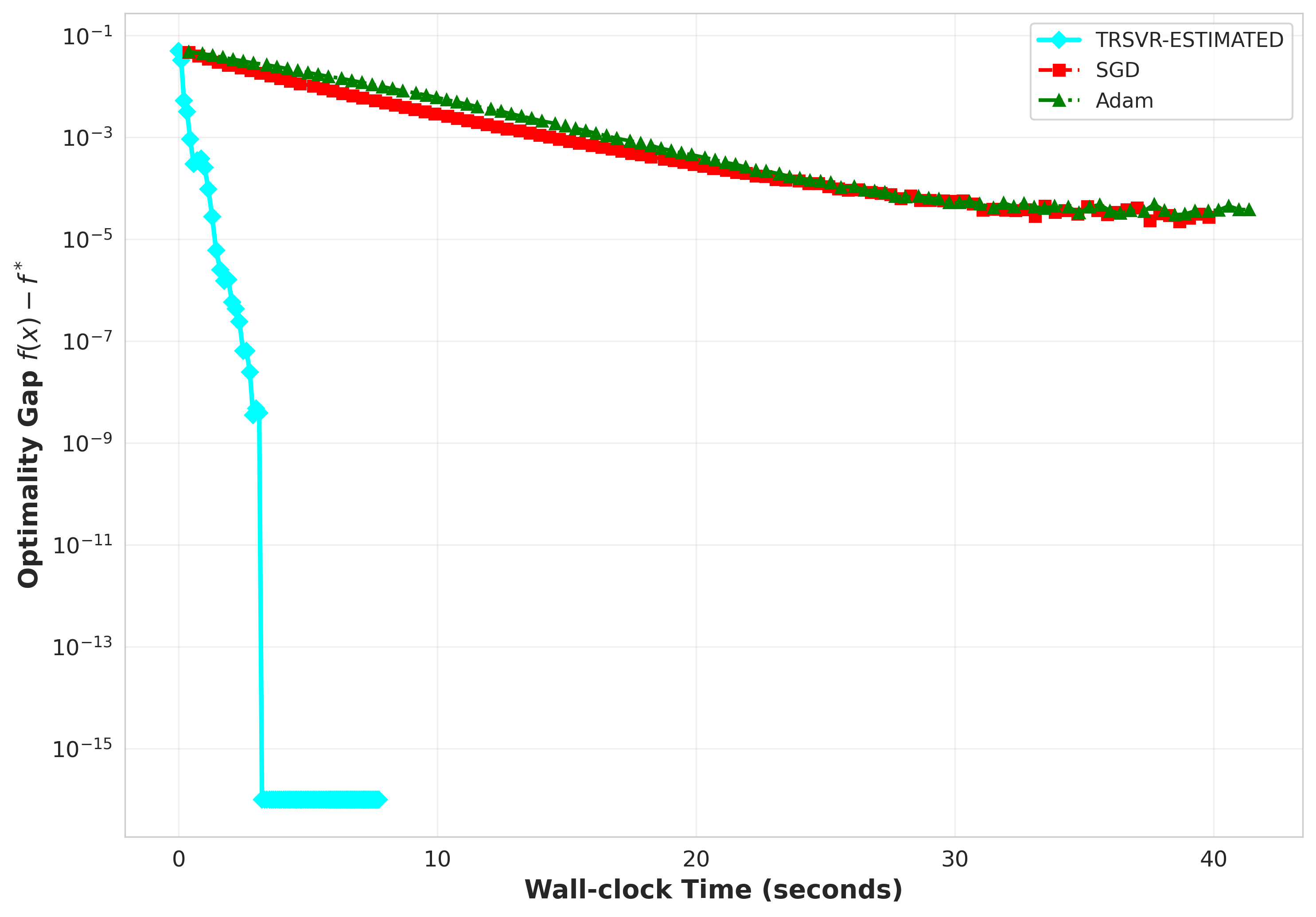}
\caption{Optimality gap with wall-clock time on the \texttt{IJCNN1} dataset. Each trajectory represents a method.}
\label{fig:Ijcnn1_SGD_Adam}
\end{figure}

\subsection{Sensitivity test: inner loop length vs. batch size}
\label{sec:sensitivity}

In this experiment, we investigate the trade-off between the inner loop length $S$ and the mini-batch size $b$ in the TRSVR algorithm while maintaining a constant total computational budget per epoch. We employ the objective function \eqref{obj:nonconvex} and the estimated Hessian to construct $H_{k,s}$.

We implement TRSVR-EstH on \texttt{RCV1} and \texttt{Covertype} datasets to explore high-dimensional sparse and dense problems, respectively. Throughout the experiment, we set a constant per-epoch budget $40000 = b \times S$ and vary the choice of $b$ and $S$. Specifically, we consider three regimes:\\
\textbf{Large Batch:} We choose small $S\in \{80, 50, 40, 20, 10\}$ and large $b\in\{500, 800, 1000, 2000, 4000\}$. This regime represents trust-region methods with infrequent updates.\\
\textbf{Balanced:} We choose moderate $b\in\{100, 200, 400\}$ and $S\in\{400,200, 100\}$. This represents the standard configuration used in our other experiments.\\
\textbf{High Frequency:} We choose small $b\in\{10, 20, 50\}$ and large $S\in\{4000, 2000, 1000, 800\}$. This regime represents highly stochastic trust-region methods.

\begin{figure}[t]
\centering
\includegraphics[width=0.9\linewidth]{ 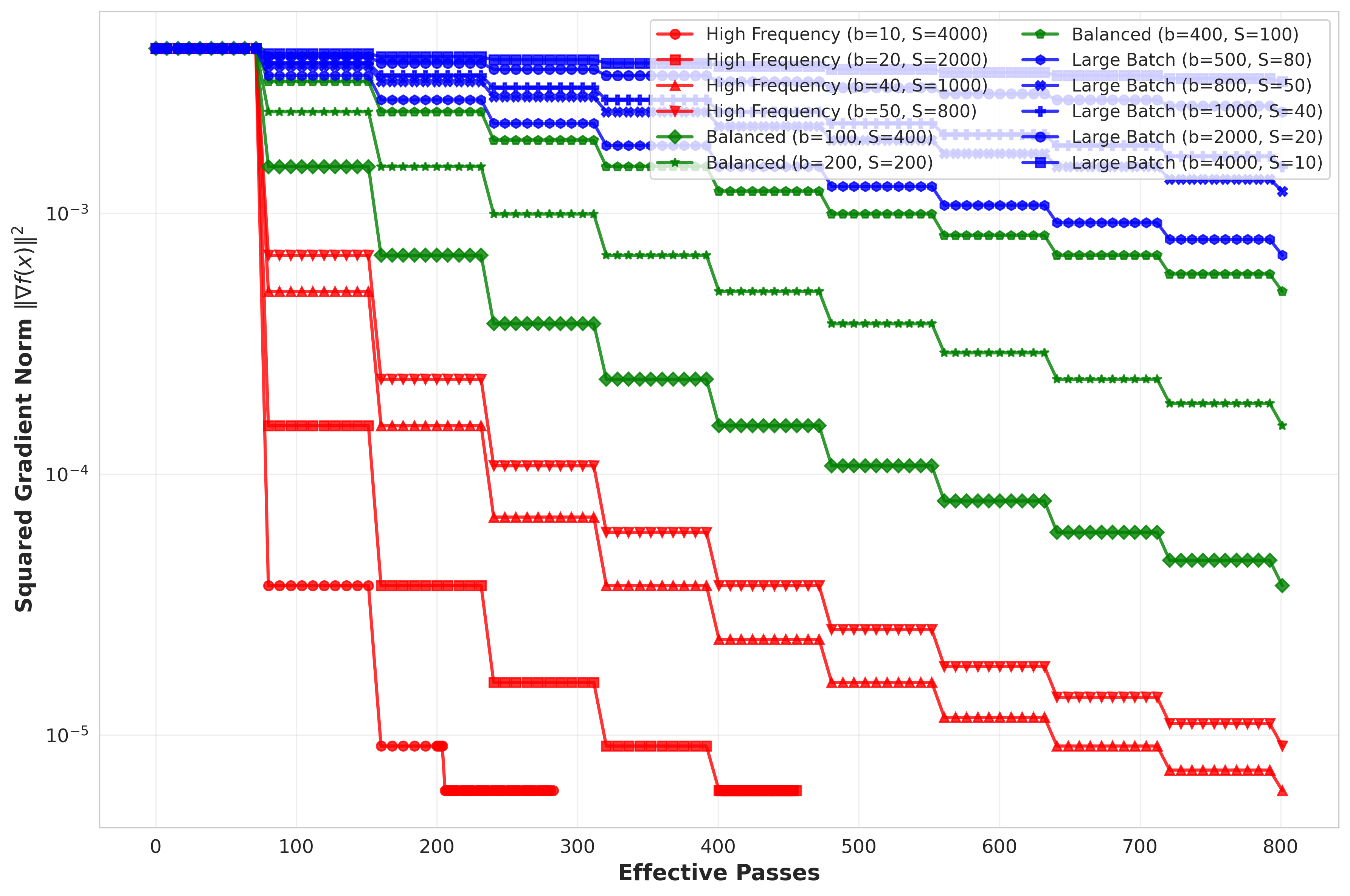}
\caption{Squared gradient norm with effective passes on the \texttt{RCV1} dataset. Each trajectory represents a combination of $b$ and $S$.}
\label{fig:RCV1_sensitivity}
\end{figure}

\begin{figure}[t]
\centering
\includegraphics[width=0.9\linewidth]{ 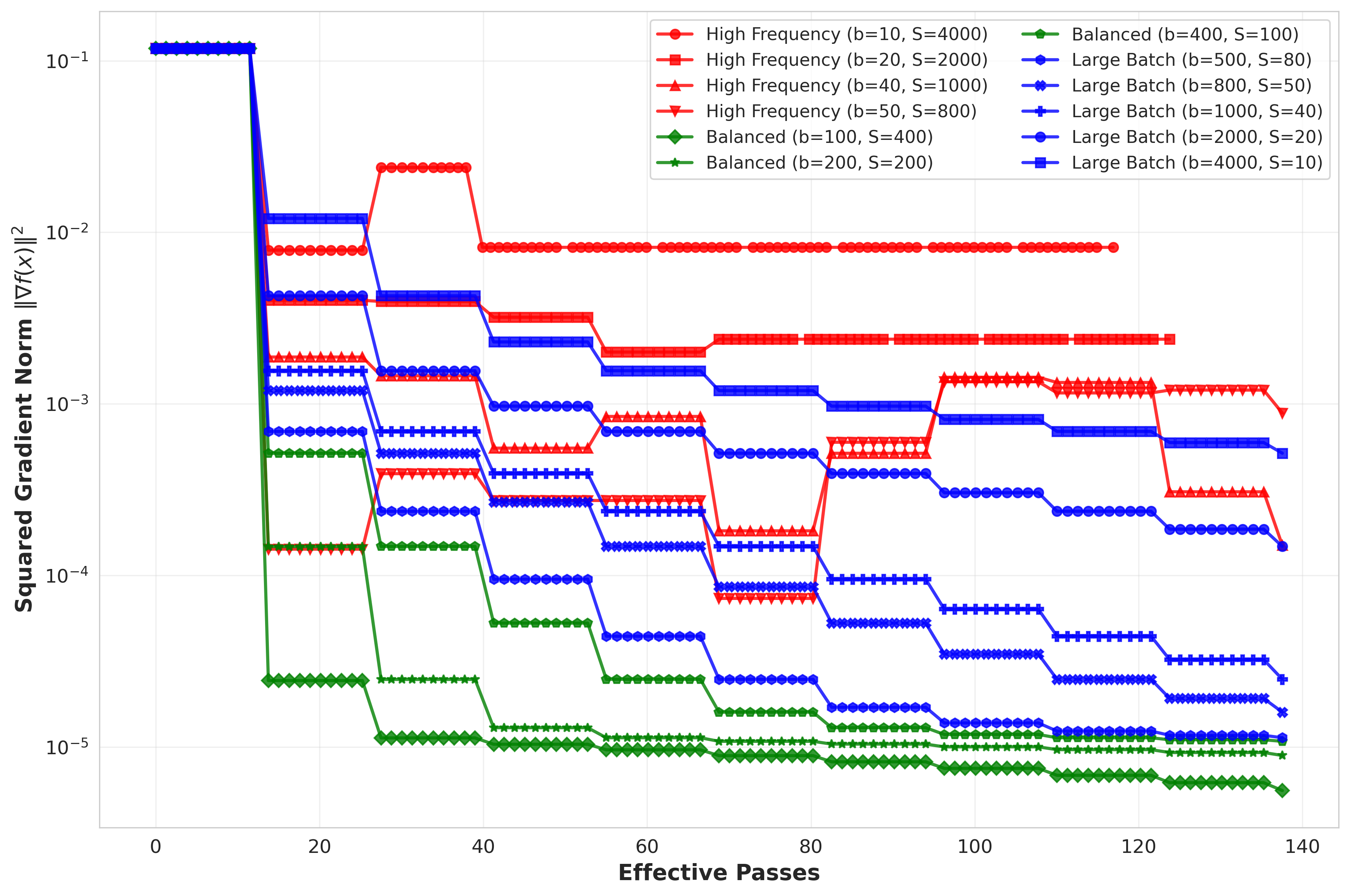}
\caption{Squared gradient norm with effective passes on the \texttt{Covertype} dataset. Each trajectory represents a combination of $b$ and $S$.}
\label{fig:Covertype_sensitivity}
\end{figure}

We observe markedly different performance patterns between sparse and dense datasets. On the sparse \texttt{RCV1} dataset (cf. Figure~\ref{fig:RCV1_sensitivity}), the high-frequency regime achieves the best performance, whereas configurations with large batch sizes cause TRSVR to stall at relatively high error levels and exhibit slow convergence. This behavior can be attributed to the fact that frequent updates with smaller batches enable the algorithm to exploit local curvature and gradient information more effectively, which is particularly important in sparse settings. In contrast, for the dense \texttt{Covertype} dataset, Figure~\ref{fig:Covertype_sensitivity} shows that the balanced regime yields the best performance, while the high-frequency regime performs noticeably worse. This degradation is likely due to the increased variance in gradient estimates induced by overly small batch sizes in dense problems. Overall, these sensitivity results indicate that the optimal configuration of TRSVR is strongly dependent on dataset structure, and that appropriate tuning of algorithmic parameters can substantially enhance performance.

\begin{figure*}[t]
    \centering
    \includegraphics[width=0.9\linewidth]{ 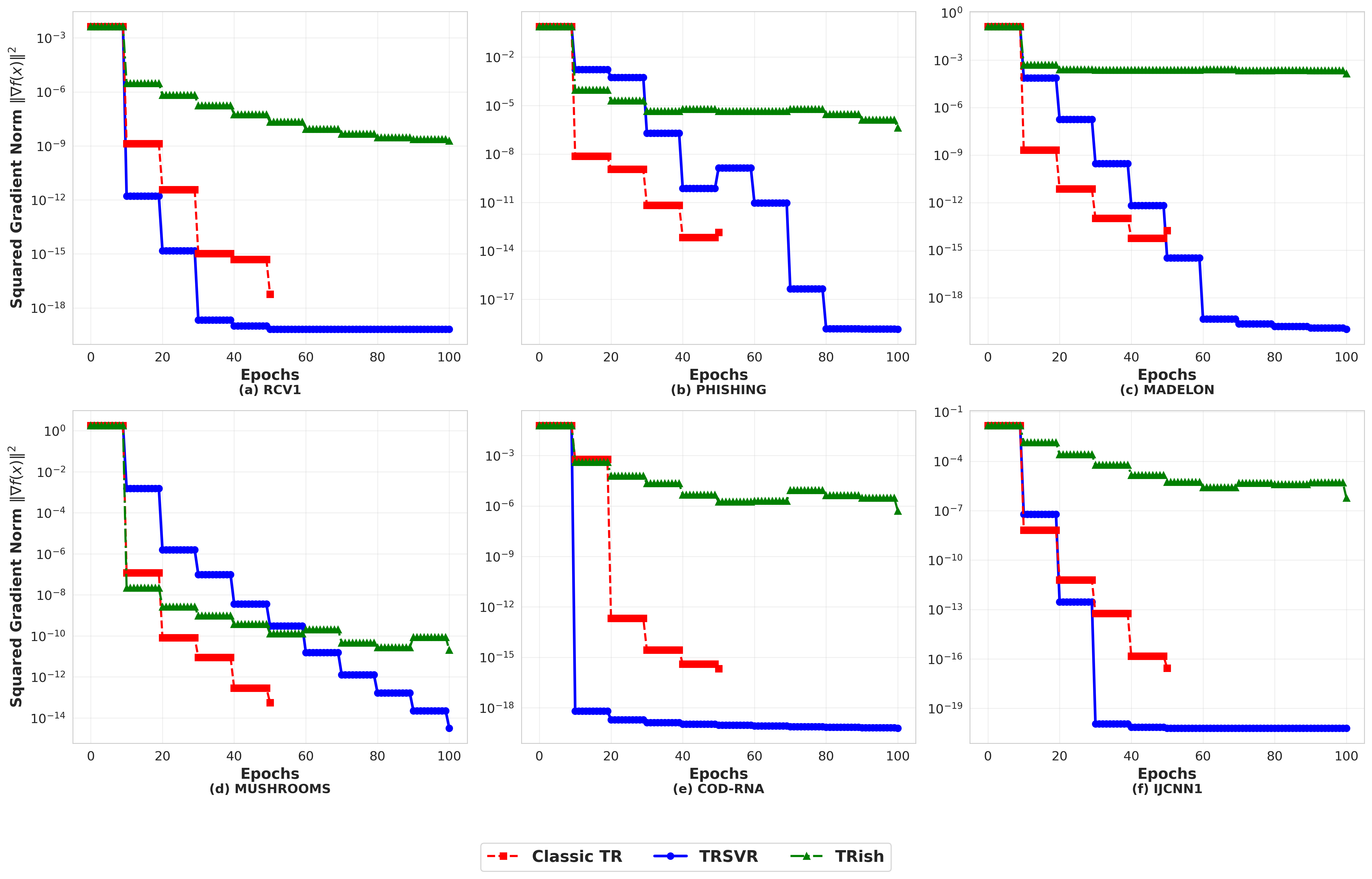}
    \caption{Squared gradient norm with epochs over six datasets. Each trajectory represents a method.}
    \label{fig:all_datasets_ablation}
\end{figure*}

\subsection{Ablation study}
\label{sec:ablation}

To further assess the effectiveness of variance reduction, we conduct an ablation study comparing TRSVR with two established trust-region algorithms: the classic trust-region method \citep[][Algorithm~4.1]{Nocedal2006Numerical} and the stochastic trust-region method (TRish) \citep{Curtis2022Fully}. For all methods, we employ the estimated Hessian to construct $H_{k,s}$.

We selected six benchmark datasets \texttt{RCV1}, \texttt{Phishing}, \texttt{Madelon}, \texttt{Mushroom}, \texttt{Cod-RNA} and \texttt{IJCNN1}. All experiments use the objective function \eqref{obj:nonconvex}. Figure~\ref{fig:all_datasets_ablation} reports the trajectories of squared gradient norm with epochs.

Across all datasets, TRSVR consistently converges to the highest-precision solutions, whereas TRish stagnates at relatively low accuracy. These results demonstrate that incorporating variance reduction substantially enhances the ability of stochastic trust-region methods to attain high-precision solutions. Moreover, compared with the classic trust-region method, TRSVR achieves comparable or higher accuracy and, in several cases, exhibits faster convergence. Notably, while the classic trust-region method requires computing the full gradient at every iteration, TRSVR computes full gradients only periodically. This highlights the significant computational savings offered by TRSVR without sacrificing optimization performance.

\section{Discussion and Future Work}

Our analysis demonstrates that SVRG-level sample complexity can be preserved even when the Hessian approximation is stochastic and potentially correlated with the gradient estimator. This finding broadens the scope of variance reduction theory beyond first-order methods. Empirically, TRSVR proves particularly effective in ill-conditioned and nonconvex settings, where first-order variance-reduced methods tend to converge slowly or plateau at suboptimal solutions. By incorporating curvature information, TRSVR exhibits superior stability and achieves faster convergence to high-precision solutions.

Despite the theoretical grounding and empirical performance of TRSVR, limitations remain.
A key limitation of TRSVR is the additional computational cost associated with solving trust-region subproblems, though this overhead can be mitigated through matrix-free Hessian-vector products and inexact solvers. Furthermore, the current framework does not incorporate acceleration techniques such as momentum or adaptive preconditioning, which could potentially enhance practical performance.

Several directions remain for future work. These include integrating momentum or adaptive curvature scaling into the TRSVR framework, extending the analysis to more general stochastic oracle models (e.g., biased or heavy-tailed noise), and developing adaptive strategies for selecting batch sizes and inner-loop lengths. Applying TRSVR to large-scale deep learning and reinforcement learning problems also presents an interesting avenue for future research.

\section{Conclusion}
We proposed TRSVR, a stochastic trust-region method that integrates variance reduction techniques. Under standard assumptions, we established global convergence guarantees and demonstrated that the iteration and sample complexity of TRSVR for finding first-order stationary points in nonconvex finite-sum problems match those of SVRG. Importantly, our analysis shows that such complexity guarantees can be preserved even when incorporating stochastic second-order geometry, thereby extending the scope of existing variance reduction theory. Extensive numerical experiments on both convex and nonconvex problems demonstrate that TRSVR converges faster and more stably to high-precision solutions than first-order baselines and existing stochastic trust-region methods. These results indicate that TRSVR provides a practical and theoretically grounded approach for large-scale stochastic optimization.

\bibliographystyle{icml2026}
\bibliography{ref}
\newpage
\appendix
\onecolumn
\section{Useful lemmas}\label{append_A}
\begin{lemma}\label{lemma:A1}
For any random variable $z$, we have
    \begin{equation*}
        \mathbb{E}[\|z-\mathbb{E}[z]\|^2] \leq \mathbb{E}[\|z\|^2].
    \end{equation*}
\end{lemma}
\begin{proof}
We have
    \begin{align*}
        \mathbb{E}[\|z-\mathbb{E}[z]\|^2] & = \mathbb{E}[\|z\|^2-2z^T\mathbb{E}[z] + \|\mathbb{E}[z] \|^2] \\
        & = \mathbb{E}[\|z\|^2] - 2 \|\mathbb{E}[z]\|^2 + \|\mathbb{E}[z] \|^2 =  \mathbb{E}[\|z\|^2] -  \|\mathbb{E}[z]\|^2 \leq \mathbb{E}[\|z\|^2].
    \end{align*}
    We thus finish the proof.
\end{proof}

\begin{lemma}\label{lemma:A2}
    For independent mean-zero random variables $z_1, z_2,\cdots, z_n$, we have
\begin{equation*}
   \mathbb{E}[\|\sum_{i=1}^n z_i \|^2] = \sum_{i=1}^n\mathbb{E}[\|z_i\|^2] .
\end{equation*}
\end{lemma}

\begin{proof}
    We have
    \begin{align*}
        \mathbb{E}[\|\sum_{i=1}^n z_i \|^2] = \sum_{i=1}^n \mathbb{E}[\|z_i \|^2] + \sum_{i\neq j} \mathbb{E}[z_i^Tz_j].
    \end{align*}
    Since $z_i,z_j$ are independent with mean zero, we have $\mathbb{E}[z_i^Tz_j]=0$ for all $i\neq j$. We thus complete the proof.
\end{proof}

\section{Proof of TRSVR}\label{append_B}

\begin{lemma}[One-Step Reduction]\label{lem:One-Step Reduction}
Under Assumption \ref{assump1}, for any iteration $(k,s)\in \mathbb{N}\times [\tilde{S}]$, we have
\begin{equation*}
 f(\bx_{k,s+1}) - f(\bx_{k,s}) \leq -\|\barg_{k,s}\|\Delta_{k,s} + \frac{1}{2}\|H_{k,s}\|\Delta_{k,s}^2 + \|\barg_{k,s} - g_{k,s} \|\Delta_{k,s} + \frac{1}{2}(L + \kappa_H)\Delta_{k,s}^2.   
\end{equation*}
\end{lemma}

\begin{proof}
By the Lipschitz continuity of $g(\bx)$ (cf. Assumption \ref{assump1}), we have
\begin{equation*}
f(\bx_{k,s+1}) \leq f(\bx_{k,s}) + g_{k,s}^T\Delta \bx_{k,s} + \frac{1}{2}L\|\Delta \bx_{k,s}\|^2.
\end{equation*}
Therefore,
\begin{align*}
   f(\bx_{k,s+1}) - f(\bx_{k,s}) - \barg_{k,s}^T\Delta \bx_{k,s} - \frac{1}{2}\Delta \bx_{k,s}^T H_{k,s}\Delta \bx_{k,s} \leq  (g_{k,s} - \barg_{k,s})^T\Delta \bx_{k,s} + \frac{1}{2}L\|\Delta \bx_{k,s}\|^2 - \frac{1}{2}\Delta \bx_{k,s}^T H_{k,s}\Delta \bx_{k,s}.
\end{align*}
Using the fact that $\|H_{k,s}\|\leq \kappa_H$, we have
\begin{align*}
    f(\bx_{k,s+1}) - f(\bx_{k,s}) - \barg_{k,s}^T\Delta \bx_{k,s} - \frac{1}{2}\Delta \bx_{k,s}^T H_{k,s}\Delta \bx_{k,s} & \leq (g_{k,s} - \barg_{k,s})^T\Delta \bx_{k,s} + \frac{1}{2}(L+\kappa_H)\|\Delta \bx_{k,s}\|^2\\
    & \leq \|g_{k,s} - \barg_{k,s}\| \|\Delta \bx_{k,s} \| + \frac{1}{2}(L+\kappa_H)\|\Delta \bx_{k,s}\|^2 \\
    & \leq \|g_{k,s} - \barg_{k,s}\| \|\Delta_{k,s} \| + \frac{1}{2}(L+\kappa_H)\|\Delta_{k,s}\|^2.
\end{align*}
The second inequality is by Cauchy-Schwarz inequality and the last inequality is by the fact that $\|\Delta x_{k,s}\| \leq \Delta_{k,s}$. Rearranging the terms, we have
\begin{align*}
    f(\bx_{k,s+1}) - f(\bx_{k,s}) & \leq  \barg_{k,s}^T\Delta \bx_{k,s} + \frac{1}{2}\Delta \bx_{k,s}^T H_{k,s}\Delta \bx_{k,s} + \|g_{k,s} - \barg_{k,s}\| \|\Delta_{k,s} \| + \frac{1}{2}(L+\kappa_H)\|\Delta_{k,s}\|^2.
\end{align*}
Combining the conclusion in Lemma \ref{lemma:Cauchy reduction} with the above display, we have
\begin{align*}
    f(\bx_{k,s+1}) - f(\bx_{k,s}) & \leq  -\|\barg_{k,s}\| \|\Delta_{k,s}\| + \frac{1}{2}\|H_{k,s} \| \Delta_{k,s}^2 + \|\barg_{k,s} - g_{k,s} \| \Delta_{k,s} + \frac{1}{2}(L+\kappa_H)\Delta_{k,s}^2.
\end{align*}
We thus finish the proof.
\end{proof}

\begin{lemma}[One-Step Expected Reduction]\label{lem: One-Step Expected Reduction}
Under Assumptions \ref{assump1} and \ref{assump2}, suppose
$\alpha \leq \frac{1}{2(L + 2\kappa_H)}$ holds, then for any iteration $(k,s)\in \mathbb{N}\times [\tilde{S}]$,  we have
\begin{equation*}
     \mathbb{E}_{k,s}[f(\bx_{k,s+1})] - f(\bx_{k,s}) \leq -\frac{1}{4}\alpha\|g_{k,s}\|^2 + \frac{1}{2}(L + 2\kappa_H)\alpha^2\mathbb{E}_{k,s}[\| \barg_{k,s} - g_{k,s}\|^2].
\end{equation*}
\end{lemma}

\begin{proof}
Recall that $\Delta_{k,s} = \alpha\|\barg_{k,s}\|$ by the algorithm design, which combined with the conclusion in Lemma \ref{lem:One-Step Reduction} and $\|H_{k,s}\| \leq \kappa_H$ yields
\begin{align*}
f(\bx_{k,s+1}) - f(\bx_{k,s}) \leq -\alpha\|\barg_{k,s}\|^2  + \alpha\|\barg_{k,s} - g_{k,s} \| \|\barg_{k,s}\| + \frac{1}{2}(L + 2 \kappa_H)\alpha^2\|\barg_{k,s}\|^2.
\end{align*}
By Young's inequality (i.e., $ab \leq \frac{1}{2}a^2 + \frac{1}{2}b^2$), we have
\begin{align*}
f(\bx_{k,s+1}) - f(\bx_{k,s}) &\leq -\alpha\|\barg_{k,s}\|^2  + \frac{1}{2}\alpha\|\barg_{k,s} - g_{k,s} \|^2 + \frac{1}{2}\alpha\|\barg_{k,s}\|^2 + \frac{1}{2}(L + 2 \kappa_H)\alpha^2\|\barg_{k,s}\|^2 \\
& = - \frac{1}{2} \alpha\|\barg_{k,s}\|^2  + \frac{1}{2}\alpha\|\barg_{k,s} - g_{k,s} \|^2  + \frac{1}{2}(L + 2 \kappa_H)\alpha^2\|\barg_{k,s}\|^2.
\end{align*}
Since $\bar{g}_{k,s}$ is an unbiased estimator of $g_{k,s}$, we have $ \mathbb{E}_{k,s}[\|\bar{g}_{k,s}\|^2] = \mathbb{E}_{k,s}[\|g_{k,s} - \bar{g}_{k,s}\|^2] + \|g_{k,s}\|^2$. Therefore, after taking expectation conditional on $\bx_{k,s}$, we have
\begin{equation*}
\mathbb{E}_{k,s}[f(\bx_{k,s+1})] - f(\bx_{k,s}) \leq -\frac{1}{2}\alpha\|g_{k,s}\|^2 + \frac{1}{2}(L+ 2\kappa_H)\alpha^2\|g_{k,s}\|^2 + \frac{1}{2}(L+ 2\kappa_H)\alpha^2\mathbb{E}_{k,s}[\|g_{k,s} - \barg_{k,s}\|^2].
\end{equation*}

Since $\alpha \leq \frac{1}{2(L+ 2\kappa_H)}$, we have $ \frac{1}{2}(L+ 2\kappa_H)\alpha^2\|g_{k,s}\|^2 \leq \frac{1}{4}\alpha\|g_{k,s}\|^2$, thus
\begin{equation*}
     \mathbb{E}_{k,s}[f(\bx_{k,s+1})] - f(\bx_{k,s}) \leq -\frac{1}{4}\alpha\|g_{k,s}\|^2 + \frac{1}{2}(L + 2\kappa_H)\alpha^2\mathbb{E}_{k,s}[\|\bar{g}_{k,s} - g_{k,s} \|^2].
\end{equation*}
We complete the proof.
\end{proof}

\begin{lemma}[Upper Bound of Variance]\label{lem:Upper Bound of Variance}
Let the variance-reduced gradient estimate $\barg_{k,s}$ be computed as in \eqref{var_reduced_grad}. Then for all $(k,s) \in \mathbb{N} \times [\tilde{S}]$, we have
\begin{equation*}
    \mathbb{E}_{k,s}[\|\barg_{k,s} - g_{k,s}\|^2] \leq \frac{L^2}{b}\|\bx_{k,s} - \bx_{k,0}\|^2.
\end{equation*}
\end{lemma}
\begin{proof}
Denote $J_{k,s} = \tilde{g}_{k,s}- \tilde{g}_{k,0}$, then the unbiasedness condition in Assumption \ref{assump2} implies that $\mathbb{E}_{k,s}[J_{k,s}] = g_{k,s} - g_{k,0}$. Recall that $\barg_{k,s}=\tilde{g}_{k,s} - (\tilde{g}_{k,0} - g_{k,0})$, we then have
\begin{align*}
    \mathbb{E}_{k,s}[\|\barg_{k,s} - g_{k,s}\|^2] &= \mathbb{E}_{k,s}[\|\tilde{g}_{k,s} - (\tilde{g}_{k,0} - g_{k,0}) - g_{k,s}\|^2]\\
    & =\mathbb{E}_{k,s}[\|J_{k,s} + g_{k,0} - g_{k,s}\|^2] \\
    &= \mathbb{E}_{k,s}[\|J_{k,s} - \mathbb{E}_{k,s}[J_{k,s}]\|^2]. 
\end{align*}
Note that $J_{k,s} = \frac{1}{b}\sum_{i\in I_{k,s}}(\nabla f_i(\bx_{k,s})-\nabla f_i(\bx_{k,0}))$, and $\mathbb{E}_{k,s}[\nabla f_i(\bx_{k,s})-\nabla f_i(\bx_{k,0})] = \mathbb{E}_{k,s}[J_{k,s}]$ for all $i\in [N]$, we have 
\begin{align*}
    \mathbb{E}_{k,s}[\|\barg_{k,s} - g_{k,s}\|^2] &= \frac{1}{b^2}\mathbb{E}_{k,s}[\|\sum_{i\in I_{k,s}}(\nabla f_i(\bx_{k,s}) - \nabla f_i(\bx_{k,0}) - \mathbb{E}_{k,s}[J_{k,s}])\|^2] \\
    &= \frac{1}{b^2}\sum_{i\in I_{k,s}}\mathbb{E}_{k,s}[\|\nabla f_i(\bx_{k,s}) - \nabla f_i(\bx_{k,0}) - \mathbb{E}_{k,s}[J_{k,s}]\|^2] \\
    &\leq \frac{1}{b^2}\sum_{i\in I_{k,s}} \mathbb{E}_{k,s}[\|\nabla f_i(\bx_{k,s}) - \nabla f_i(\bx_{k,0})\|^2] \\
    &\leq \frac{L^2}{b}\|\bx_{k,s} - \bx_{k,0}\|^2,
\end{align*}
where the second equality is by Lemma \ref{lemma:A2}, the first inequality is by Lemma \ref{lemma:A1}, and the last inequality is by the Lipschitz continuity of the component gradients (cf. Assumption \ref{assump1}).
\end{proof}

\begin{theorem}[Global Convergence]\label{thm:Global Convergence}
Under Assumptions \ref{assump1} and \ref{assump2}, we select
$\alpha = \frac{\mu_0 b}{2(L+2\kappa_H)N^\gamma}\in(0,1]$, $b = \mu_1 N^{\gamma}$ with $\mu_0,\mu_1,\gamma \in(0,1]$, and $ S \leq \bigg\lfloor\frac{N^{3\gamma/2}}{\mu_0(b+\frac{\mu_0L^2b}{2(L+2\kappa_H)})}\bigg\rfloor$, then for all $k\geq 0$ and $s\in [\tilde{S}]$, there exist universal constants $\mu_0$ and $v_0 \in (0, 1)$ such that
\begin{equation*}
    \mathbb{E}\left[\frac{1}{(K+1)S}\sum_{k=0}^K\sum_{s=0}^{S-1}\|g_{k,s}\|^2\right] \leq \frac{2(L+\kappa_H)(\mathbb{E}[f(\bx_{0,0})] - f_{\inf})}{(K+1)S\mu_0\mu_1 v_0}.
\end{equation*}

\end{theorem}

\begin{proof}
Combining the results of Lemma \ref{lem: One-Step Expected Reduction} and Lemma \ref{lem:Upper Bound of Variance}, we have
\begin{equation*}
\mathbb{E}_{k,s}[f(\bx_{k,s+1})] - f(\bx_{k,s}) \leq -\frac{1}{4}\alpha\|g_{k,s}\|^2 + \frac{L^2(L + 2\kappa_H)}{2b} \alpha^2 \|\bx_{k,s} - \bx_{k,0}\|^2. 
\end{equation*}
Next we analyze the term $\|\bx_{k,s} - \bx_{k,0}\|^2$. Note that
\begin{align*}
    \mathbb{E}_{k,s}[\|\bx_{k,s+1} - \bx_{k,0}\|^2] &= \mathbb{E}_{k,s}[\|\bx_{k,s+1} - \bx_{k,s} + \bx_{k,s} - \bx_{k,0}\|^2] \\
    &= \mathbb{E}_{k,s}[\|\bx_{k,s+1} - \bx_{k,s}\|^2]  + 2\mathbb{E}_{k,s}[(\bx_{k,s+1} - \bx_{k,s})^T(\bx_{k,s} - \bx_{k,0})] + \|\bx_{k,s} - \bx_{k,0}\|^2\\
    & \leq  \mathbb{E}_{k,s}[\|\bx_{k,s+1} - \bx_{k,s}\|^2]  + 2\mathbb{E}_{k,s}[\|\bx_{k,s+1} - \bx_{k,s}\|]\|\bx_{k,s} - \bx_{k,0}\| + \|\bx_{k,s} - \bx_{k,0}\|^2\\
     & \leq  \mathbb{E}_{k,s}[\Delta_{k,s}^2]  + 2\mathbb{E}_{k,s}[\Delta_{k,s}]\|\bx_{k,s} - \bx_{k,0}\| + \|\bx_{k,s} - \bx_{k,0}\|^2\\
    & = \alpha^2\mathbb{E}_{k,s}[\|\barg_{k,s}\|^2]  + 2\alpha\mathbb{E}_{k,s}[\|\barg_{k,s}\|]\|\bx_{k,s} - \bx_{k,0}\| + \|\bx_{k,s} - \bx_{k,0}\|^2 \\
    & \leq \alpha^2\mathbb{E}_{k,s}[\|\barg_{k,s}\|^2] + \frac{\alpha}{z}\mathbb{E}_{k,s}[\|\barg_{k,s}\|^2] + \alpha z \|\bx_{k,s} - \bx_{k,0}\|^2 + \|\bx_{k,s} - \bx_{k,0}\|^2\\
    & = \left(1 + \frac{1}{\alpha z}\right)\alpha^2\mathbb{E}_{k,s}[\|\barg_{k,s}\|^2] + (1 + \alpha z) \|\bx_{k,s} - \bx_{k,0}\|^2.
\end{align*}
Recall that $\mathbb{E}_{k,s}[\|\bar{g}_{k,s}\|^2] = \|g_{k,s}\|^2 + \mathbb{E}_{k,s}[\|g_{k,s}-\bar{g}_{k,s}\|^2] \leq \|g_{k,s}\|^2 + L^2/b\cdot \|\bx_{k,s} - \bx_{k,0}\|^2$, we have
\begin{align*}
  \mathbb{E}_{k,s}[\|\bx_{k,s+1} - \bx_{k,0}\|^2]  & \leq \left(1+\frac{1}{\alpha z}\right)\alpha^2\|g_{k,s}\|^2 + \left(1+\alpha z+\left(\alpha^2 + \frac{\alpha}{ z}\right)\frac{ L^2}{b}\right)\|\bx_{k,s} - \bx_{k,0}\|^2.
\end{align*}
To further the analysis, we define the following Lyapunov function $\Phi_{k,s} = f(\bx_{k,s}) + \lambda_s\|\bx_{k,s} - \bx_{k,0}\|^2$, where
\begin{align*}
    \lambda_s = \frac{\alpha^2 L^2(L+ 2\kappa_H)}{2b} + \lambda_{s+1}\left(1+\alpha z+\left(\alpha^2 + \frac{\alpha}{ z}\right)\frac{ L^2}{b}\right)
\end{align*}
and $\lambda_S=0$. Note that $\Phi_{k,0} = f(\bx_{k,0})$.
Based on the above definition, we find
\begin{align*}
   \mathbb{E}_{k,s}[\Phi_{k,s+1} - \Phi_{k,s}] & = \mathbb{E}_{k,s}[f(\bx_{k,s+1}) - f(\bx_{k,s})] + \lambda_{s+1} \mathbb{E}_{k,s}[\|\bx_{k,s+1} - \bx_{k,0}\|^2] - \lambda_s\|\bx_{k,s} - \bx_{k,0}\|^2 \\
   & \leq -\frac{1}{4}\alpha\|g_{k,s}\|^2 + \frac{L^2(L + 2\kappa_H)}{2b} \alpha^2 \|\bx_{k,s} - \bx_{k,0}\|^2  - \lambda_s\|\bx_{k,s} - \bx_{k,0}\|^2 \\
   & \quad + \lambda_{s+1}\left(1+\frac{1}{\alpha z}\right)\alpha^2\|g_{k,s}\|^2 + \left(1+\alpha z+\left(\alpha^2 + \frac{\alpha}{ z}\right)\frac{ L^2}{b}\right) \|\bx_{k,s} - \bx_{k,0}\|^2\\
   & = -\left(\frac{1}{4}\alpha - \lambda_{s+1}\left(1+\frac{1}{\alpha z}\right)\alpha^2\right)\|g_{k,s}\|^2.
\end{align*}
Denoting
\begin{equation*}
    \Lambda_s = \frac{1}{4}\alpha - \lambda_{s+1}\left(1+\frac{1}{\alpha z}\right)\alpha^2, 
\end{equation*}
and setting $\Lambda_{\min} = \min_{s\in[\tilde{S}]} \Lambda_s$, we have $\mathbb{E}_{k,s} [\Phi_{k,s+1} - \Phi_{k,s}] \leq -\Lambda_{\min}\|g_{k,s}\|^2$. Rearranging the terms, we have
\begin{equation*}
\|g_{k,s}\|^2 \leq \frac{\mathbb{E}_{k,s}[\Phi_{k,s} - \Phi_{k,s+1}]}{\Lambda_{\min}}.
\end{equation*}
Taking full expectation and summing over $s = 0,\dots,S-1$, we have
\begin{equation*}
    \sum_{s=0}^{S-1}\mathbb{E}[\|g_{k,s}\|^2] \leq \frac{\mathbb{E}[\Phi_{k,0}] - \mathbb{E}[\Phi_{k,S}]}{\Lambda_{\min}} = \frac{\mathbb{E}[f(\bx_{k,0}) - f(\bx_{k+1,0})]}{\Lambda_{\min}}, 
\end{equation*}
where we use the facts that $\Phi_{k,0} = f(\bx_{k,0}) $ and $\lambda_S=0$. Summing over $k = 0,1,2,\cdots,K$ for any finite $K$, we have
\begin{equation*}
    \sum_{k=0}^K\sum_{s=0}^{S-1}\mathbb{E}[\|g_{k,s}\|^2] \leq \frac{\mathbb{E}[f(\bx_{0,0})] - f_{\inf}}{\Lambda_{\min}}.
\end{equation*}
We then divide both sides by $(K+1)S$ and yield
\begin{equation}\label{thm:conv}
     \mathbb{E}\left[\frac{1}{(K+1)S}\sum_{k=0}^K\sum_{s=0}^{S-1}\|g_{k,s}\|^2\right] \leq \frac{\mathbb{E}[f(\bx_{0,0})] - f_{\inf}}{(K+1)S \cdot \Lambda_{\min}}.
\end{equation}
Next we lower bound $\Lambda_{\min}$ by setting parameters properly. Specifically, we set $\mu_0\in(0,1]$ as a parameter to be specified later, $\gamma\in(0,1]$, and
\begin{align*}
    \alpha = \frac{\mu_0 b}{2(L+2\kappa_H)N^\gamma},\quad z = \frac{2(L+2\kappa_H)}{N^{\gamma/2}}, \quad b = N^{\gamma},\quad S \leq \bigg\lfloor\frac{N^{3\gamma/2}}{\mu_0(b+\frac{\mu_0L^2b}{2(L+2\kappa_H)})}\bigg\rfloor.
\end{align*}
By the recursive structure of $\lambda_s$, we have
\begin{equation*}
    \lambda_0=\frac{\alpha^2 L^2(L+2\kappa_H)}{2b}\cdot \frac{(1+\rho)^S-1}{\rho},
\end{equation*}
where
\begin{align*}
    \rho & = \alpha z + \left(\alpha^2 + \frac{\alpha}{ z}\right)\frac{ L^2}{b} \\
    & = \frac{\mu_0 b}{N^{3\gamma/2}} + \frac{\mu_0^2L^2 b}{4(L+2\kappa_H)^2N^{2\gamma}}+\frac{\mu_0L^2}{4(L+\kappa_H)^2N^{\gamma/2}}\\
    & = \frac{\mu_0 b}{N^{3\gamma/2}} + \frac{\mu_0^2L^2 b}{4(L+2\kappa_H)^2N^{2\gamma}}+\frac{\mu_0L^2b}{4(L+\kappa_H)^2N^{3\gamma/2}} \\
    & \leq \mu_0 N^{-3\gamma/2} \left( b + \frac{L^2b}{2(L+2\kappa_H)^2}\right).
\end{align*}
Therefore, we have
\begin{align*}
    \lambda_0 & =  \frac{\mu_0^2 L^2 b}{8(L+2\kappa_H)N^{2\gamma}} \cdot \frac{(1+\rho)^S-1}{\frac{\mu_0 b}{N^{3\gamma/2}} + \frac{\mu_0^2L^2 b}{4(L+2\kappa_H)^2N^{2\gamma}}+\frac{\mu_0L^2b}{4(L+\kappa_H)^2N^{3\gamma/2}}}\\
    & \leq \frac{\mu_0^2 L^2 b}{8(L+2\kappa_H)N^{2\gamma}} \cdot \frac{(1+\rho)^S-1}{\frac{\mu_0 b}{N^{3\gamma/2}}}\\
    & = \frac{\mu_0 L^2(e-1)}{8(L+2\kappa_H) }N^{-\gamma/2},
\end{align*}
where we also use the upper bound of $S$ and the monotonicity of the function $(1+1/l)^l$ and the fact that $(1+1/l)^l\rightarrow e$ as $l\rightarrow\infty$.
Plugging in the upper bound of $\lambda_0$ to the definition of $\Lambda_s$, noticing that $\lambda_s$ is monotone decreasing with respect to $s$, we have
\begin{align*}
    \Lambda_{\min} & =  \frac{1}{4}\alpha - \lambda_{0}\left(1+\frac{1}{\alpha z}\right)\alpha^2\\
    & \geq \frac{1}{4}\alpha - \frac{\alpha\mu_0 L^2(e-1)}{8(L+2\kappa_H) }N^{-\gamma/2}\left(\alpha+\frac{1}{ z}\right)\\
    & = \alpha \left[\frac{1}{4} - \frac{\mu_0 L^2(e-1)}{8(L+2\kappa_H) }N^{-\gamma/2}\left(\frac{\mu_0 b}{2(L+2\kappa_H)N^\gamma}+\frac{N^{\gamma/2}}{2(L+2\kappa_H)}\right)\right]\\
    & \geq \alpha \left[\frac{1}{4} - \frac{\mu_0 L^2(e-1)(\mu_0 b+1)}{16(L+2\kappa_H)^2 }\right].
\end{align*}
We denote $v_0\coloneqq \frac{1}{4} - \frac{\mu_0 L^2(e-1)(\mu_0 b+1)}{16(L+2\kappa_H)^2 }$. When $\mu_0$ is set properly, we can ensure $v_0 > 0$. Therefore, we have
\begin{equation*}
  \Lambda_{\min} \geq \frac{\mu_0bv_0}{2(L+\kappa_H)N^{\gamma}}.  
\end{equation*}
Combining the above lower bound of $\Lambda_{\min}$ with \eqref{thm:conv} and the selection of $b$ leads to the conclusion.
\end{proof}

\section{Implementation details and full results}\label{append_C}

\subsection{General implementation and hyperparameter tuning}

We implement all algorithms in Python. For nonconvex experiments involving high-dimensional data, we utilize PyTorch to leverage automatic differentiation for efficient gradient and Hessian-vector product computations. The trust-region subproblems are solved using the Steihaug-CG method \citep[][Algorithm 7.2]{Nocedal2006Numerical}. To balance computational cost and solution quality, we set the maximum number of CG iterations to 200 for experiments in Section~\ref{sec:convex} and 500 for all other experiments.

For hyperparameter tuning, we conduct grid search for all methods. For variance-reduced baselines (SVRG, SAGA, SARAH), learning rates are searched over a logarithmic grid scaled by the theoretical Lipschitz constant $L_{\max}$. For SGD and Adam, we perform a logarithmic grid search over typical operational ranges. For trust-region methods (deterministic trust-region, TRish, and TRSVR), we search over their respective key hyperparameters, including the initial trust-region radius $\Delta_0$ and the radius-control parameter $\alpha$. The optimal configurations for each experiment are detailed in the following subsections.

\subsection{Convex optimization}

We generated each entry of the initialization from a standard normal distribution and fixed TRSVR parameters: $b = 200$ and $S = 100$. In tuning the learning rate and the radius-control parameter, we used a pre-calculated Lipschitz constant $L \approx 50.1$. Table~\ref{tab:results} summarizes the optimal configurations. Results are reported in Figures \ref{fig:synthetic_gap_epochs}-\ref{fig:synthetic_grad_passes}.

\begin{table}[ht]
    \centering
    \caption{Optimal Configurations for Synthetic Data}
    \label{tab:results}
    \begin{tabular}{lc}
        \toprule
        \textbf{Method} & \textbf{Optimal Configuration} \\
        \midrule
        TRSVR-Id & $\alpha = 0.05$ \\
        TRSVR-EstH & $\alpha = 0.06$ \\
        \midrule
        SVRG & LR $= 3.04 \times 10^{-2}$ \\
        SARAH & LR $= 1.00 \times 10^{-3}$ \\
        SAGA & LR $= 1.31 \times 10^{-2}$ \\
        \bottomrule
    \end{tabular}
\end{table}

\begin{figure}[H]
    \centering
    \begin{minipage}{0.45\textwidth}
        \centering
        \includegraphics[width=\linewidth]{ 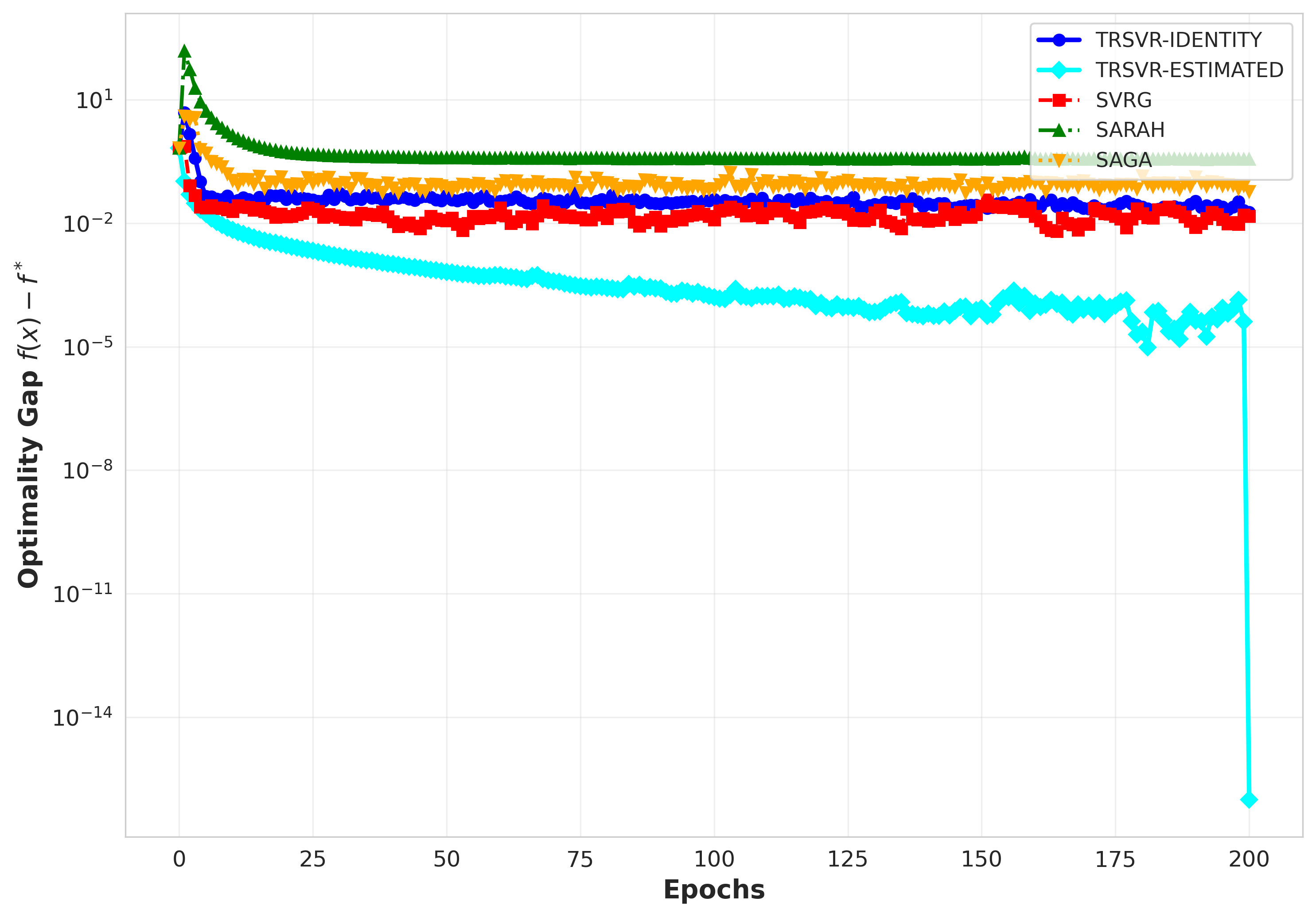}
        \caption{Optimality gap vs. epochs for synthetic data.}
        \label{fig:synthetic_gap_epochs}
    \end{minipage}\hfill
    \begin{minipage}{0.45\textwidth}
        \centering
        \includegraphics[width=\linewidth]{ Figures/synthetic/synthetic_optimality_gap_passes__SYNTHETIC_Convergence_Optimality_Gap_vs_Effective_Passes.png}
        \caption{Optimality gap vs. effective passes for synthetic data.}
        \label{fig:synthetic_gap_passes}
    \end{minipage}
    
    \vspace{0.5cm}
    
    \begin{minipage}{0.45\textwidth}
        \centering
        \includegraphics[width=\linewidth]{ 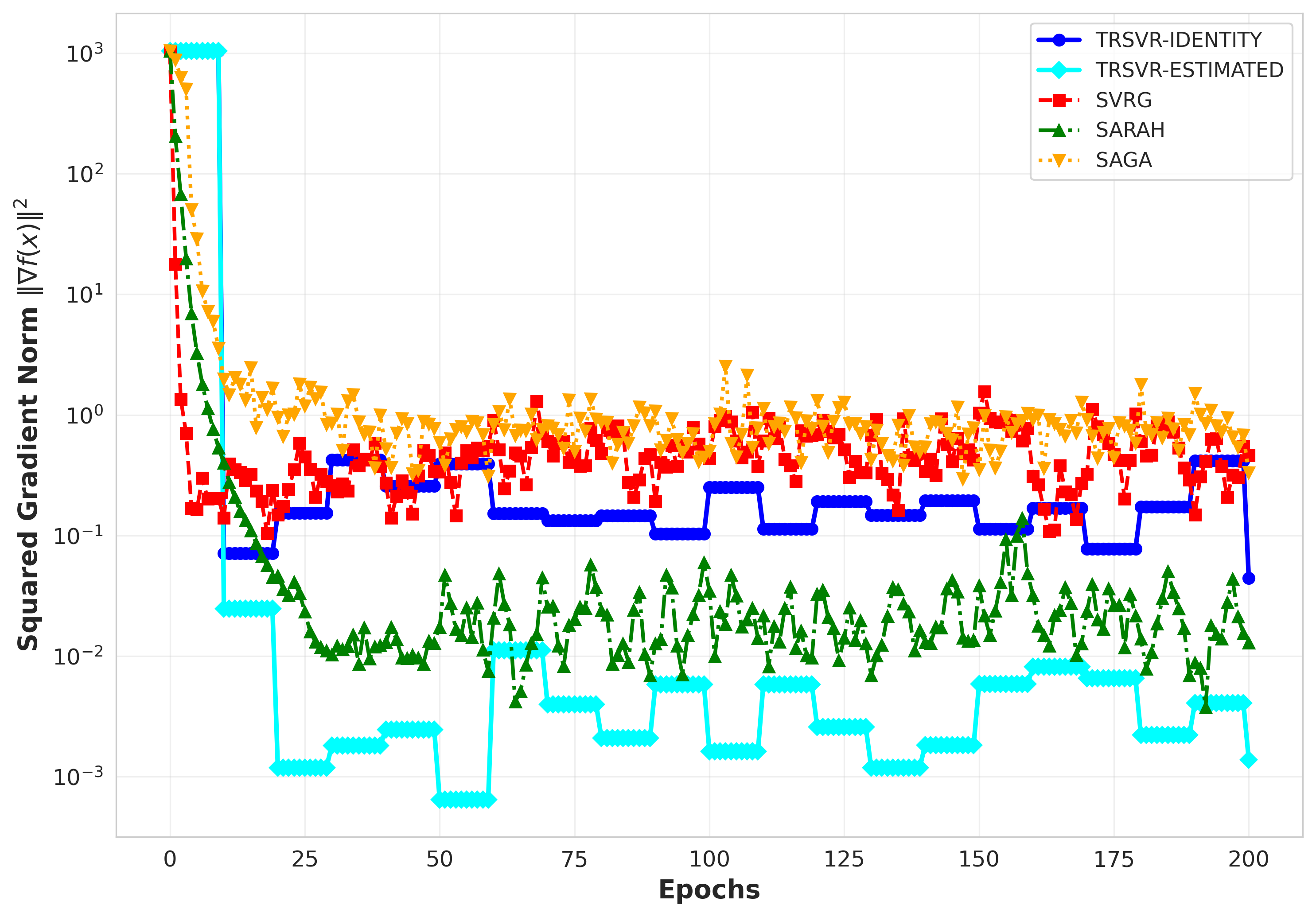}
        \caption{Squared gradient norm vs. epochs for synthetic data.}
        \label{fig:synthetic_grad_epochs}
    \end{minipage}\hfill
    \begin{minipage}{0.45\textwidth}
        \centering
        \includegraphics[width=\linewidth]{ Figures/synthetic/synthetic_grad_norm_passes__SYNTHETIC_Convergence_Squared_Gradient_Norm_vs_Effective_Passes.png}
        \caption{Squared gradient norm vs. effective passes for synthetic data.}
        \label{fig:synthetic_grad_passes}
    \end{minipage}
\end{figure}

\subsection{Nonconvex optimization}

For the TRSVR methods, we fixed the mini-batch size $b=200$ and inner-loop length $S=200$. We summarize the optimal configurations in Table~\ref{tab:optimal_params}. Results are reported in Figures \ref{fig:rcv1_gap_epochs}-\ref{fig:mushroom_grad_passes}.

\begin{table}[ht]
    \centering
    \caption{Optimal Configurations for Nonconvex Experiments}
    \label{tab:optimal_params}
    \begin{tabular}{lcccc}
        \toprule
        & \multicolumn{2}{c}{\textbf{RCV1}} & \multicolumn{2}{c}{\textbf{Mushroom}} \\
        \cmidrule(lr){2-3} \cmidrule(lr){4-5}
        \textbf{Method} & \textbf{Parameter} & \textbf{Value} & \textbf{Parameter} & \textbf{Value} \\
        \midrule
        TRSVR-Id & $\alpha$ & 0.10 & $\alpha$ & 0.08 \\
        TRSVR-EstH & $\alpha$ & 0.15 & $\alpha$ & 0.09 \\
        SAGA & Learning Rate & 0.08 & Learning Rate & 0.05 \\
        SVRG & Learning Rate & 0.045 & Learning Rate & 0.015 \\
        SARAH & Learning Rate & 0.005 & Learning Rate & 0.005 \\
        \bottomrule
    \end{tabular}
\end{table}

\begin{figure}[H]
    \centering
    \begin{minipage}{0.4\textwidth}
        \centering
        \includegraphics[width=\linewidth]{ 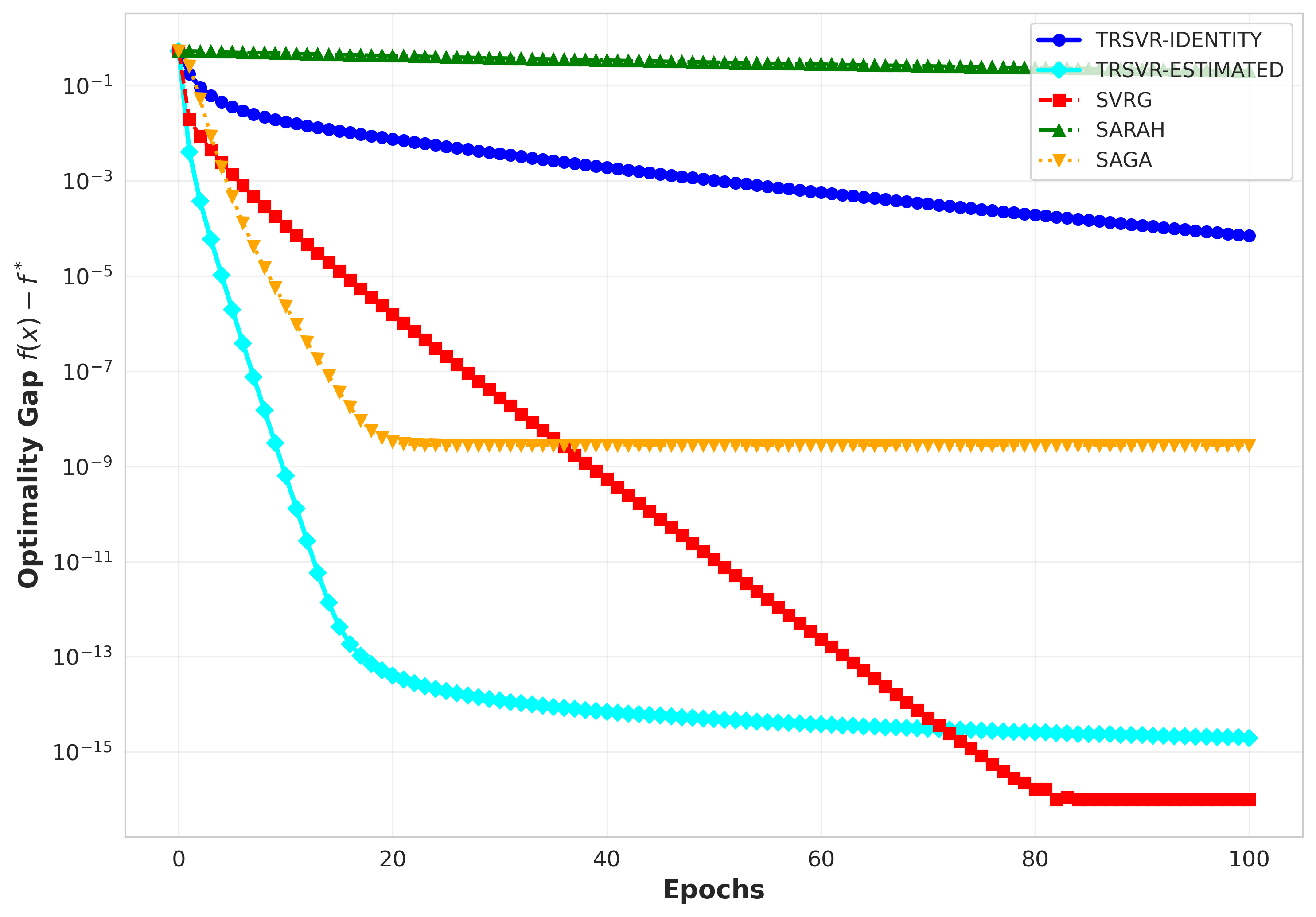}
        \caption{Optimality gap vs. epochs for RCV1.}
        \label{fig:rcv1_gap_epochs}
    \end{minipage}\hfill
    \begin{minipage}{0.4\textwidth}
        \centering
        \includegraphics[width=\linewidth]{ 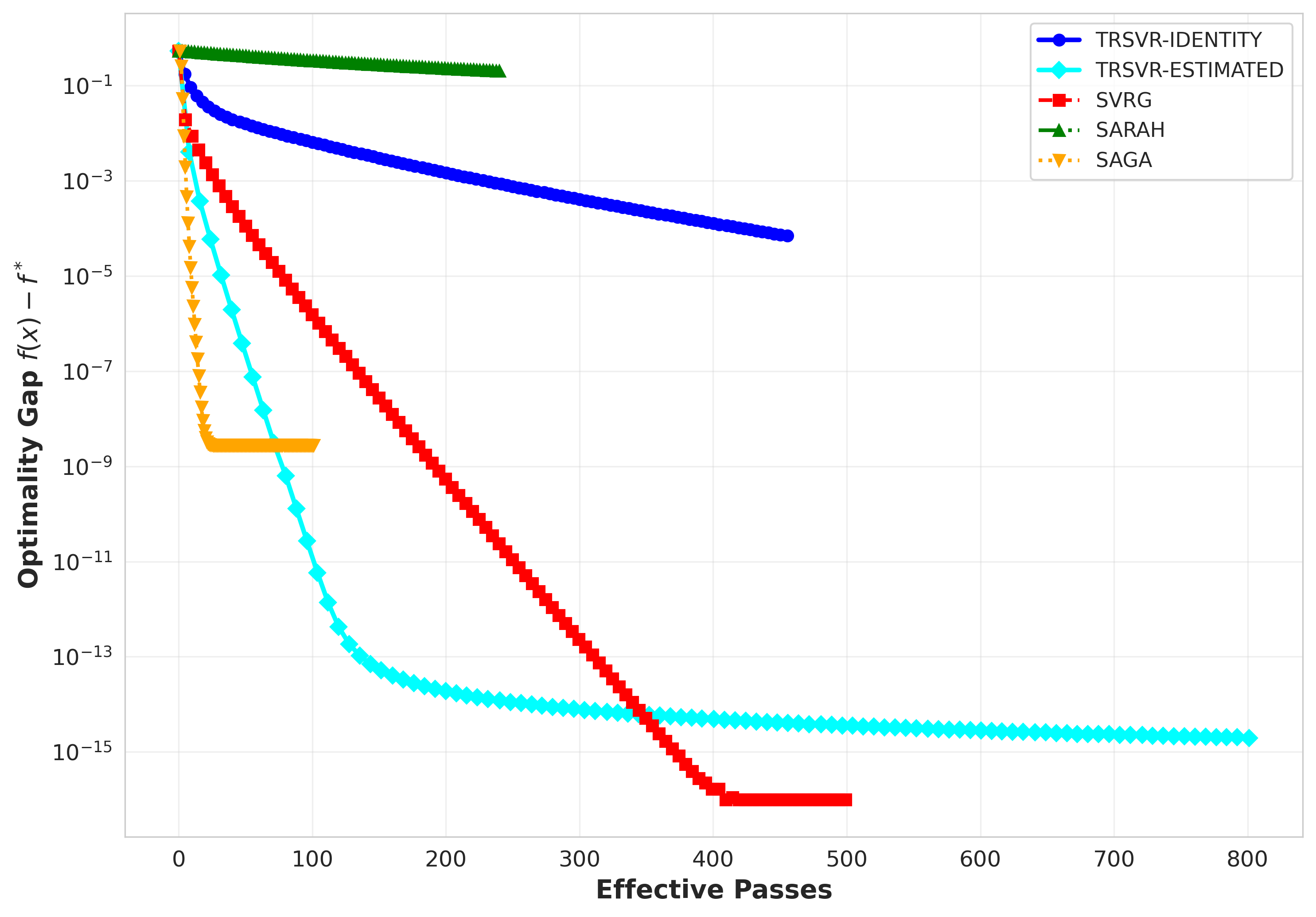}
        \caption{Optimality gap vs. effective passes for RCV1.}
        \label{fig:rcv1_gap_passes}
    \end{minipage}

    \begin{minipage}{0.4\textwidth}
        \centering
        \includegraphics[width=\linewidth]{ Figures/Stochastic_Method_Comparison/rcv1/rcv1_grad_norm_epochs__RCV1_Convergence_Squared_Gradient_Norm_vs_Epochs.png}
        \caption{Squared gradient norm vs. epochs for RCV1.}
        \label{fig:rcv1_grad_epochs}
    \end{minipage}\hfill
    \begin{minipage}{0.4\textwidth}
        \centering
        \includegraphics[width=\linewidth]{ 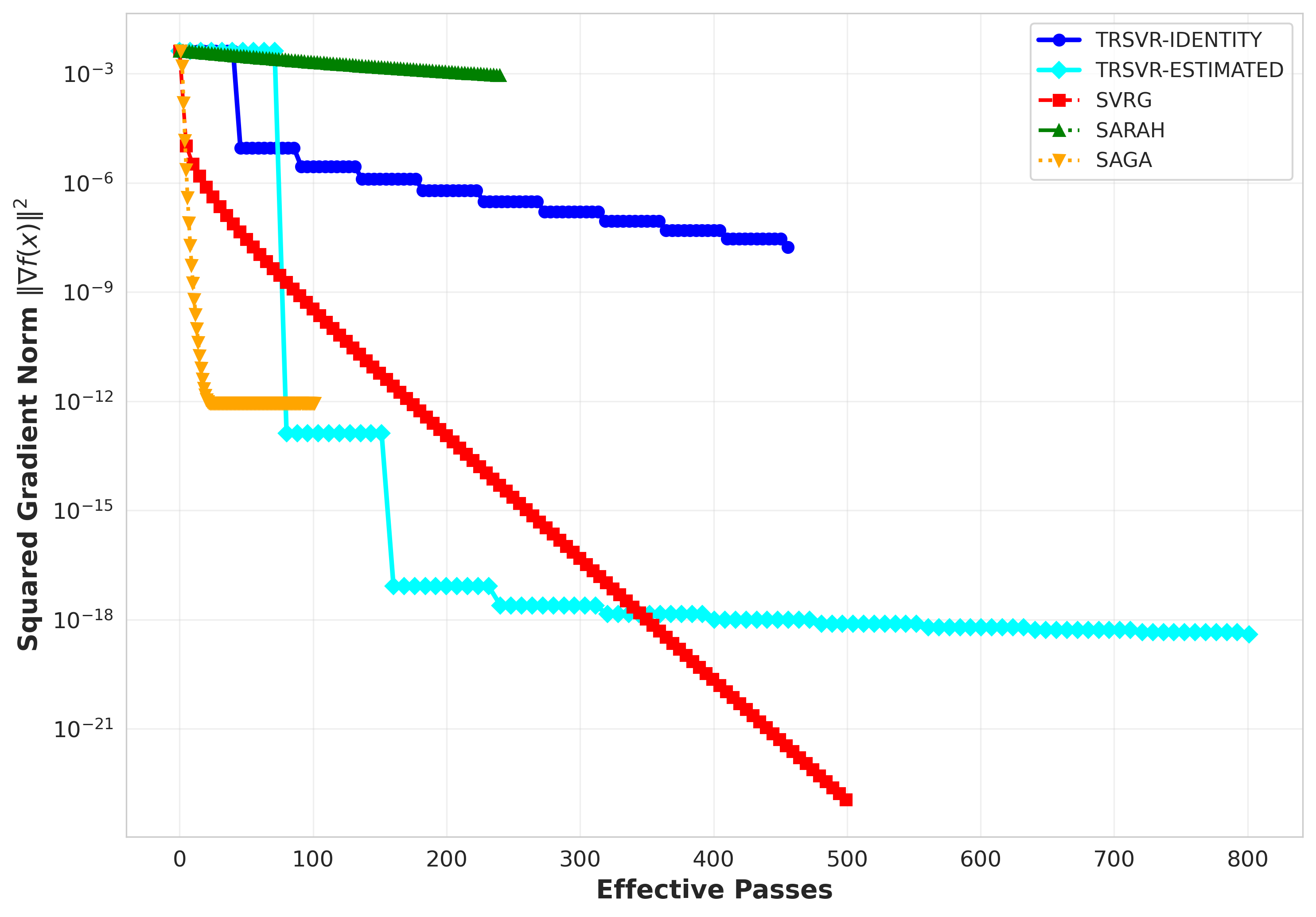}
        \caption{Squared gradient norm vs. effective passes for RCV1.}
        \label{fig:rcv1_grad_passes}
    \end{minipage}

    \begin{minipage}{0.4\textwidth}
        \centering
        \includegraphics[width=\linewidth]{ 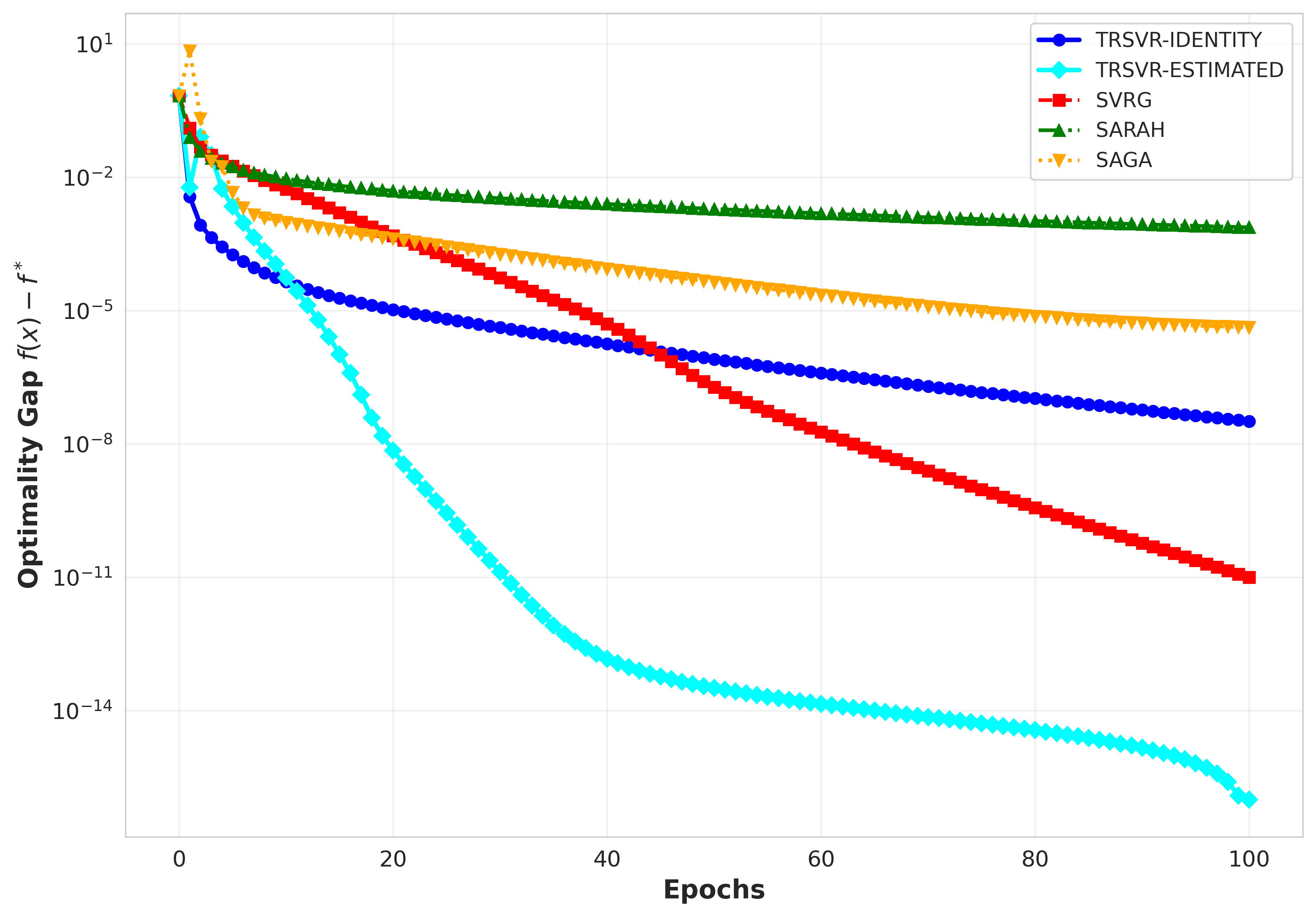}
        \caption{Optimality gap vs. epochs for Mushroom.}
        \label{fig:mushroom_gap_epochs}
    \end{minipage}\hfill
    \begin{minipage}{0.4\textwidth}
        \centering
        \includegraphics[width=\linewidth]{ 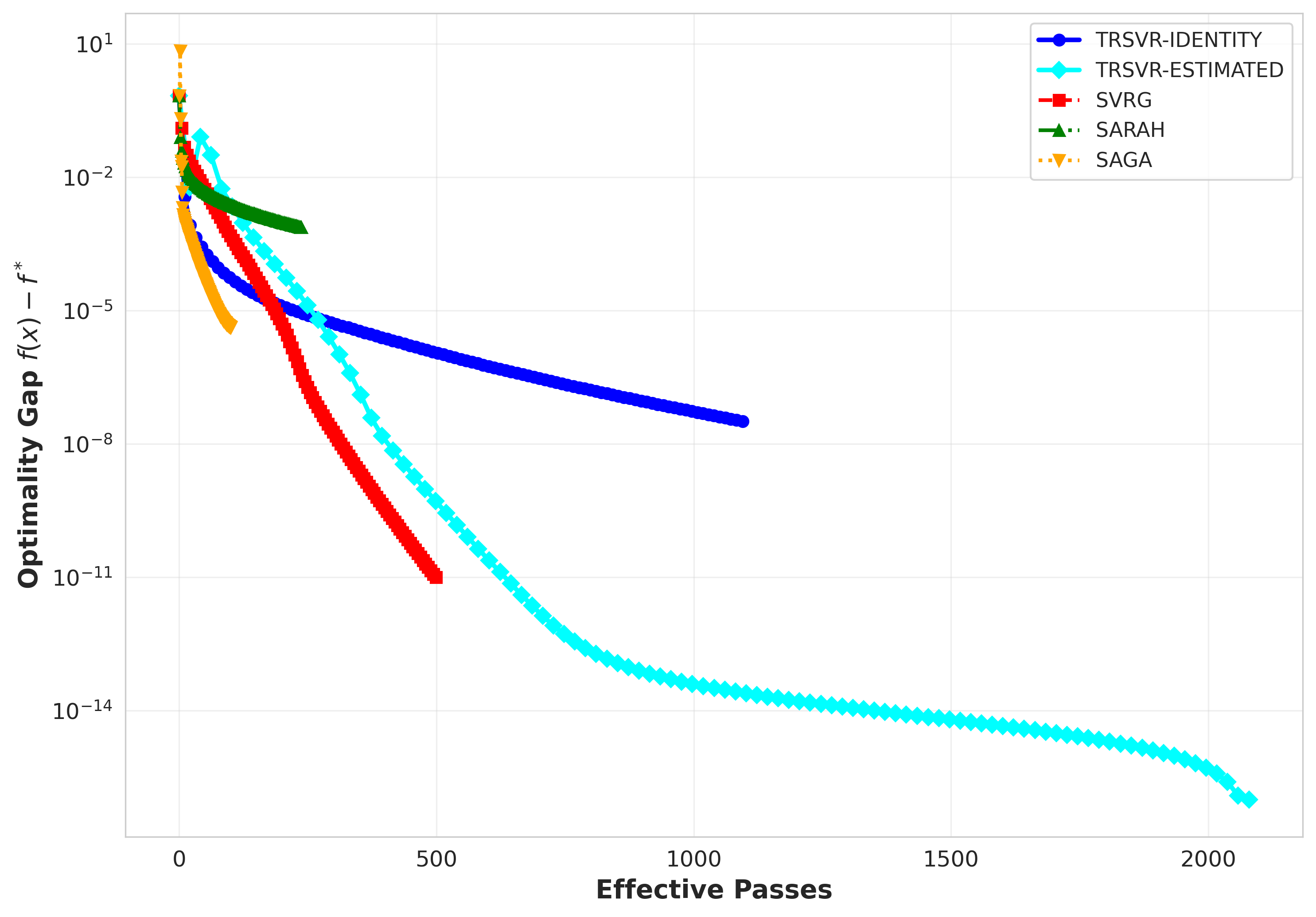}
        \caption{Optimality gap vs. effective passes for Mushroom.}
        \label{fig:mushroom_gap_passes}
    \end{minipage}

    \begin{minipage}{0.4\textwidth}
        \centering
        \includegraphics[width=\linewidth]{ Figures/Stochastic_Method_Comparison/mushroom/mushrooms_grad_norm_epochs__MUSHROOMS_Convergence_Squared_Gradient_Norm_vs_Epochs.png}
        \caption{Squared gradient norm vs. epochs for Mushroom.}
        \label{fig:mushroom_grad_epochs}
    \end{minipage}\hfill
    \begin{minipage}{0.4\textwidth}
        \centering
        \includegraphics[width=\linewidth]{ 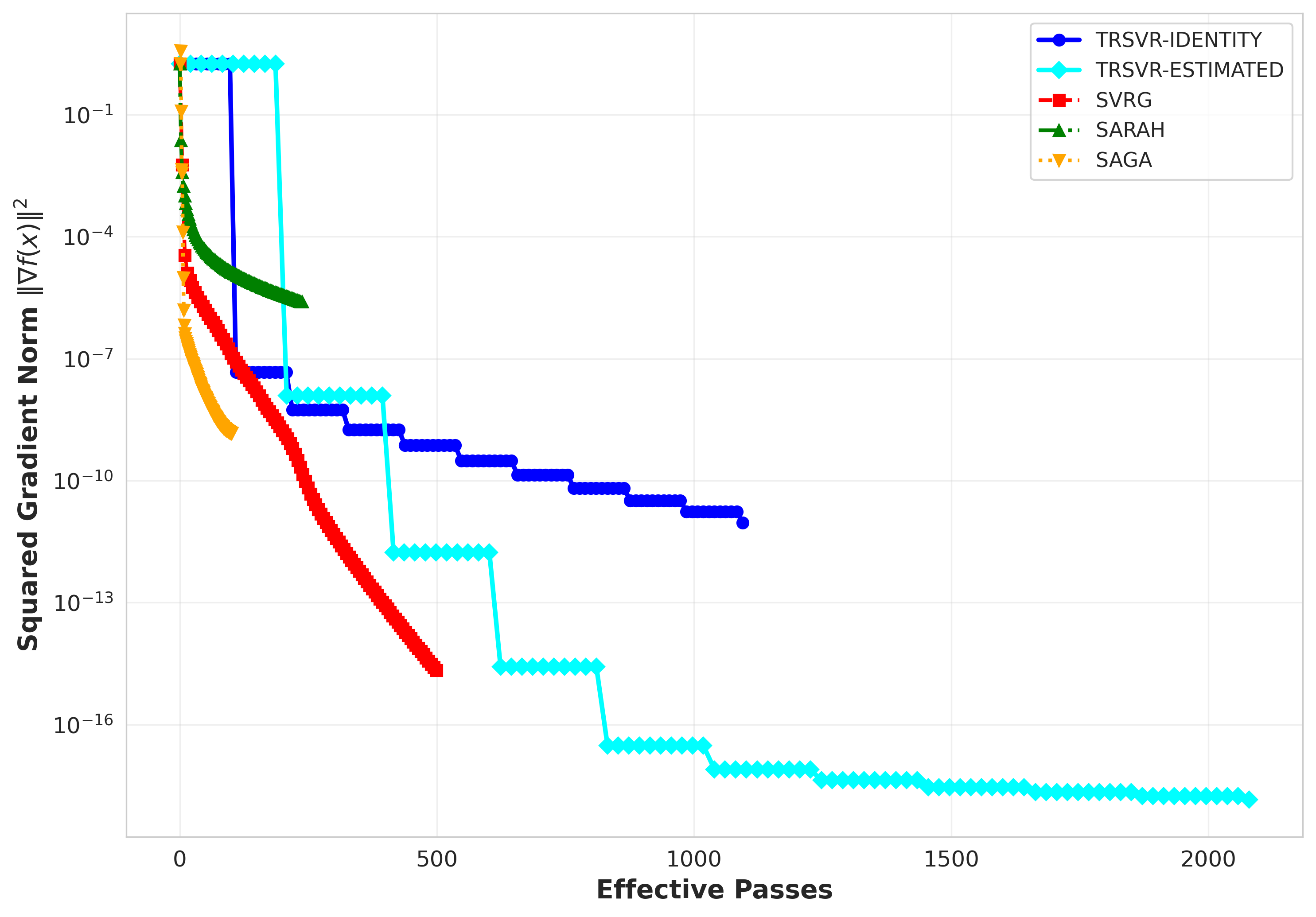}
        \caption{Squared gradient norm vs. effective passes for Mushroom.}
        \label{fig:mushroom_grad_passes}
    \end{minipage}
\end{figure}

\subsection{Comparison with SGD and Adam}

We compared TRSVR with SGD and Adam implemented in PyTorch. All algorithms were run for a fixed budget of $K=100$ epochs. We summarize the optimal configurations in Table~\ref{tab:benchmark_params}.
Results are reported in Figures \ref{fig:covertype_gap_epochs}-\ref{fig:ijcnn1_grad_passes}.

\begin{table}[h]
    \centering
    \caption{Optimal Configurations for SGD and Adam Comparison. For TRSVR, parameters denote $\alpha$ and inner-loop $(b, S)$. For SGD and Adam, parameters denote the learning rate.}
    \label{tab:benchmark_params}
    \begin{tabular}{lccc}
        \toprule
        \textbf{Dataset} & \textbf{TRSVR} & \textbf{SGD} & \textbf{Adam} \\
        \midrule
        Covertype & $\alpha=0.0043$, $(100, 400)$ & $3.7 \times 10^{-5}$ & $9.5 \times 10^{-5}$ \\
        IJCNN1 & $\alpha=0.005$, $(100, 400)$ & $2.3 \times 10^{-4}$ & $2.3 \times 10^{-4}$ \\
        \bottomrule
    \end{tabular}
\end{table}

\begin{figure}[h]
    \centering
    \begin{minipage}{0.4\textwidth}
        \centering
        \includegraphics[width=\linewidth]{ 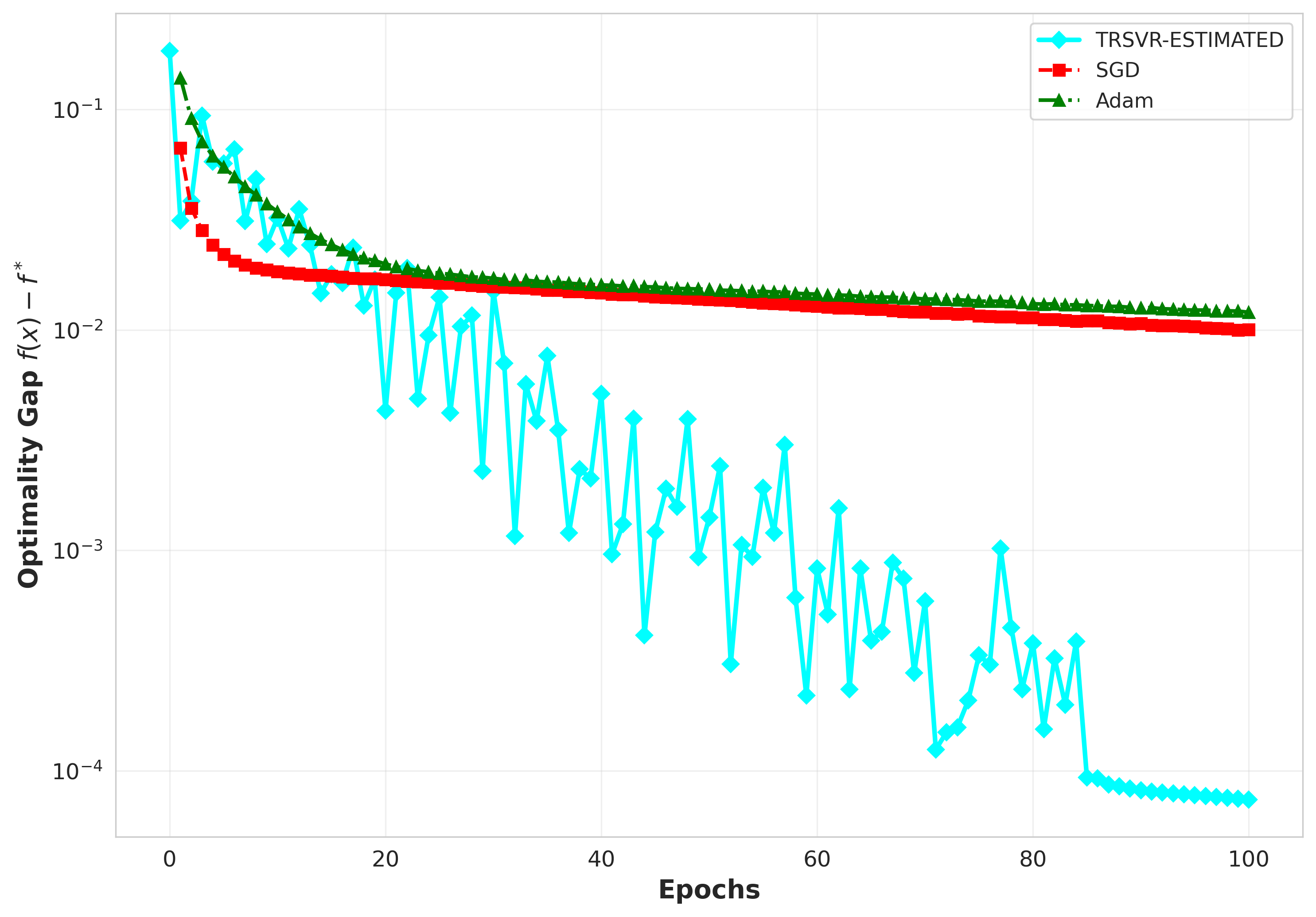}
        \caption{Optimality gap vs. epochs for Covertype.}
        \label{fig:covertype_gap_epochs}
    \end{minipage}\hfill
    \begin{minipage}{0.4\textwidth}
        \centering
        \includegraphics[width=\linewidth]{ Figures/Adam_SGD_Comparison/covertype/covertype_optimality_gap_time__COVERTYPE_Benchmark_Optimality_Gap_vs_Wall_clock_Time.png}
        \caption{Optimality gap vs. wall-clock time for Covertype.}
        \label{fig:covertype_gap_time}
    \end{minipage}
    
\end{figure}

\begin{figure}[h]
    \centering
    \begin{minipage}{0.4\textwidth}
        \centering
        \includegraphics[width=\linewidth]{ 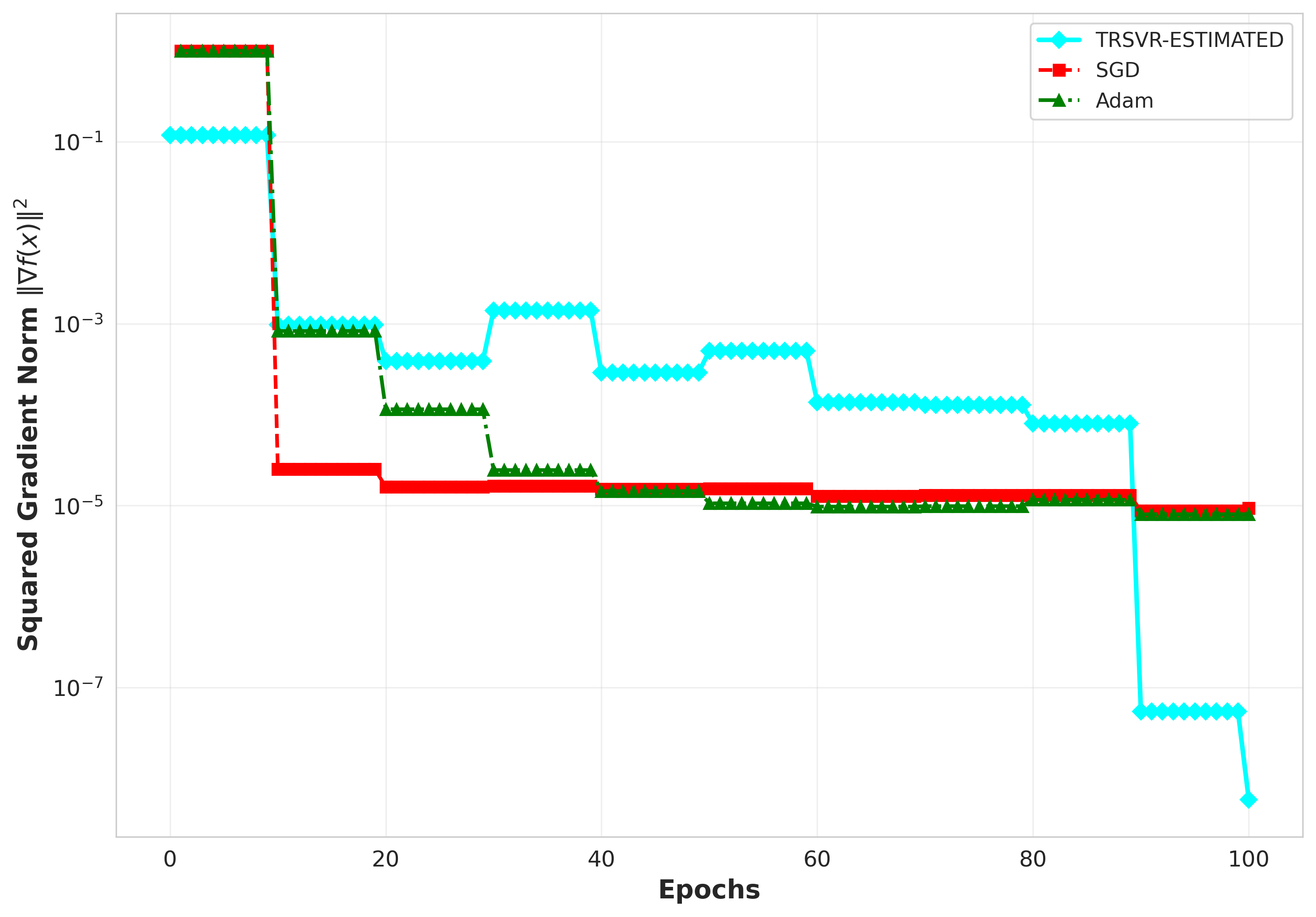}
        \caption{Squared gradient norm vs. epochs for Covertype.}
        \label{fig:covertype_grad_epochs}
    \end{minipage}\hfill
    \begin{minipage}{0.4\textwidth}
        \centering
        \includegraphics[width=\linewidth]{ 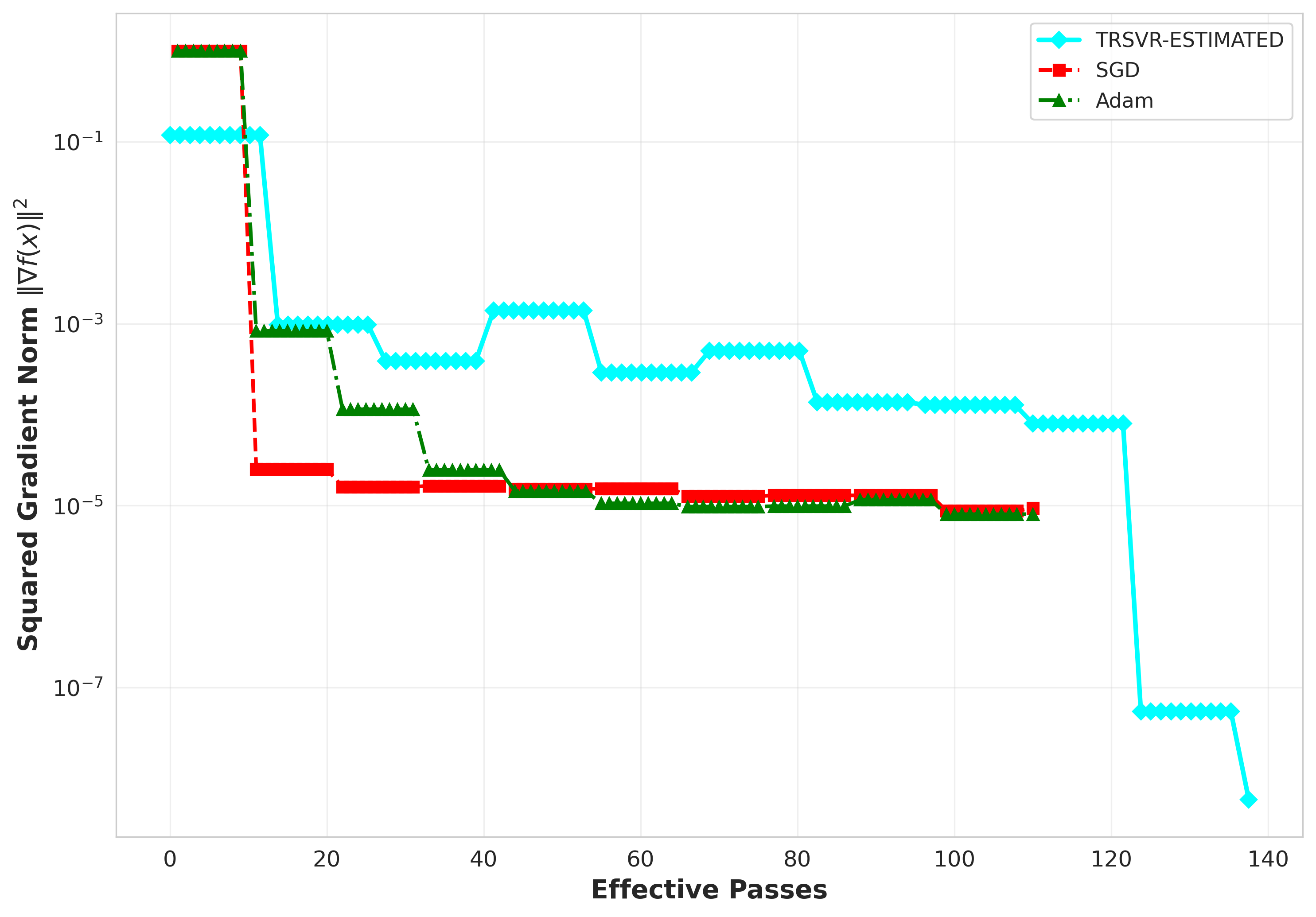}
        \caption{Squared gradient norm vs. effective passes for Covertype.}
        \label{fig:covertype_grad_passes}
    \end{minipage}
\end{figure}

\begin{figure}[h]
    \centering
    \begin{minipage}{0.4\textwidth}
        \centering
        \includegraphics[width=\linewidth]{ 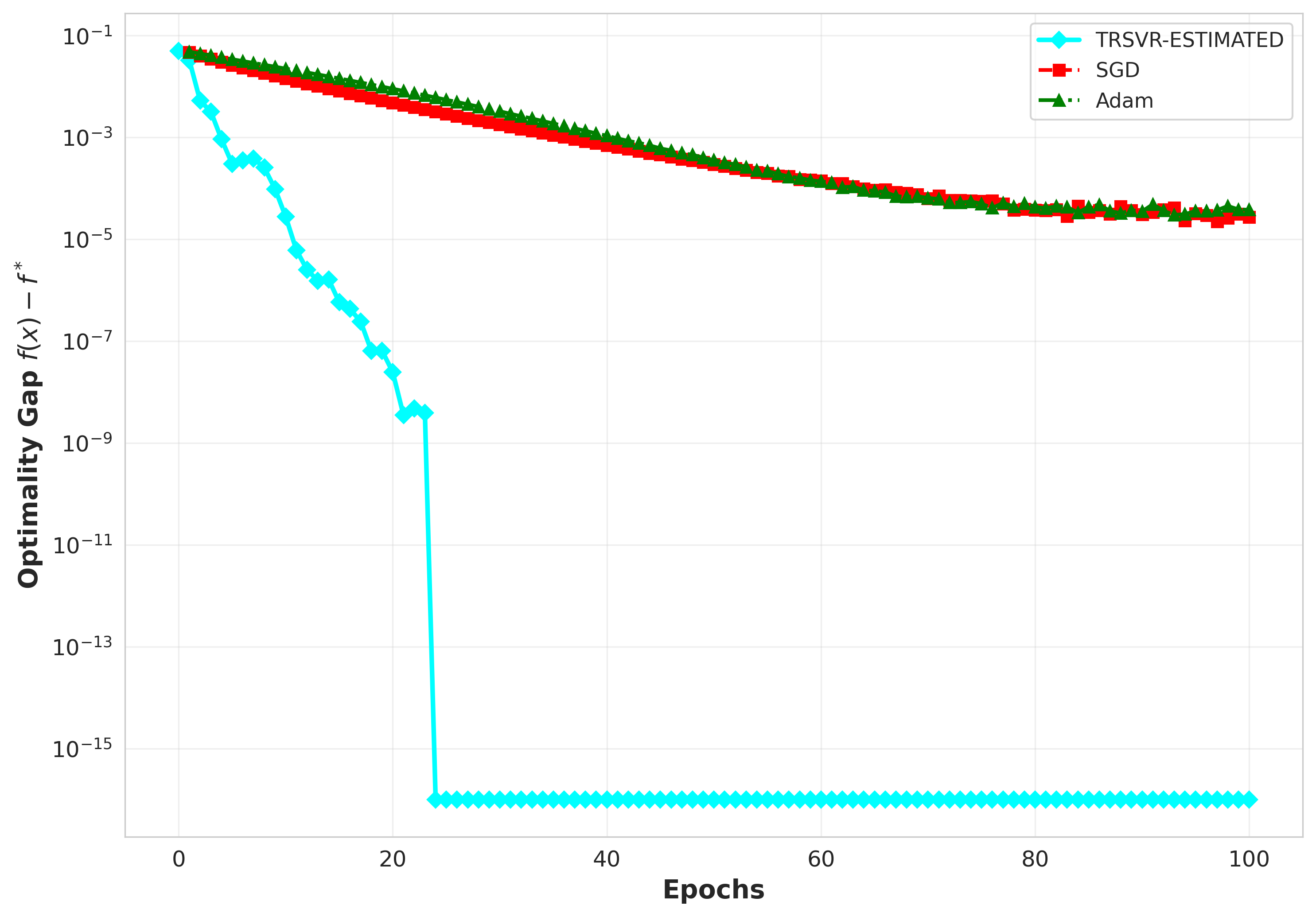}
        \caption{Optimality gap vs. epochs for IJCNN1.}
        \label{fig:ijcnn1_gap_epochs}
    \end{minipage}\hfill
    \begin{minipage}{0.4\textwidth}
        \centering
        \includegraphics[width=\linewidth]{ Figures/Adam_SGD_Comparison/ijcnn1/ijcnn1_optimality_gap_time__IJCNN1_Benchmark_Optimality_Gap_vs_Wall_clock_Time.png}
        \caption{Optimality gap vs. wall-clock time for IJCNN1.}
        \label{fig:ijcnn1_gap_time}
    \end{minipage}

\end{figure}

\begin{figure}[h]
    \centering
    
    \begin{minipage}{0.4\textwidth}
        \centering
        \includegraphics[width=\linewidth]{ 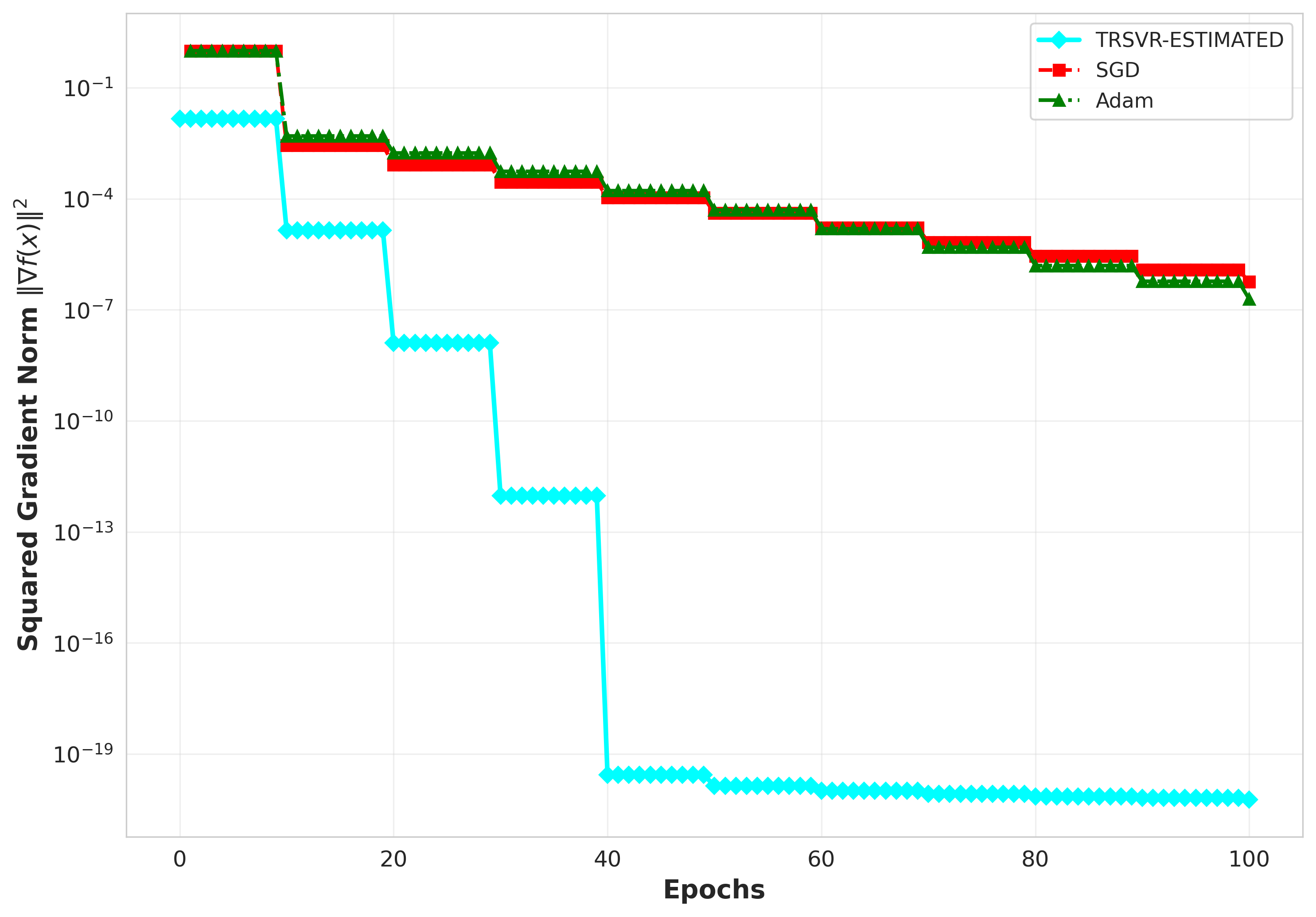}
        \caption{Squared gradient norm vs. epochs for IJCNN1.}
        \label{fig:ijcnn1_grad_epochs}
    \end{minipage}\hfill
    \begin{minipage}{0.4\textwidth}
        \centering
        \includegraphics[width=\linewidth]{ 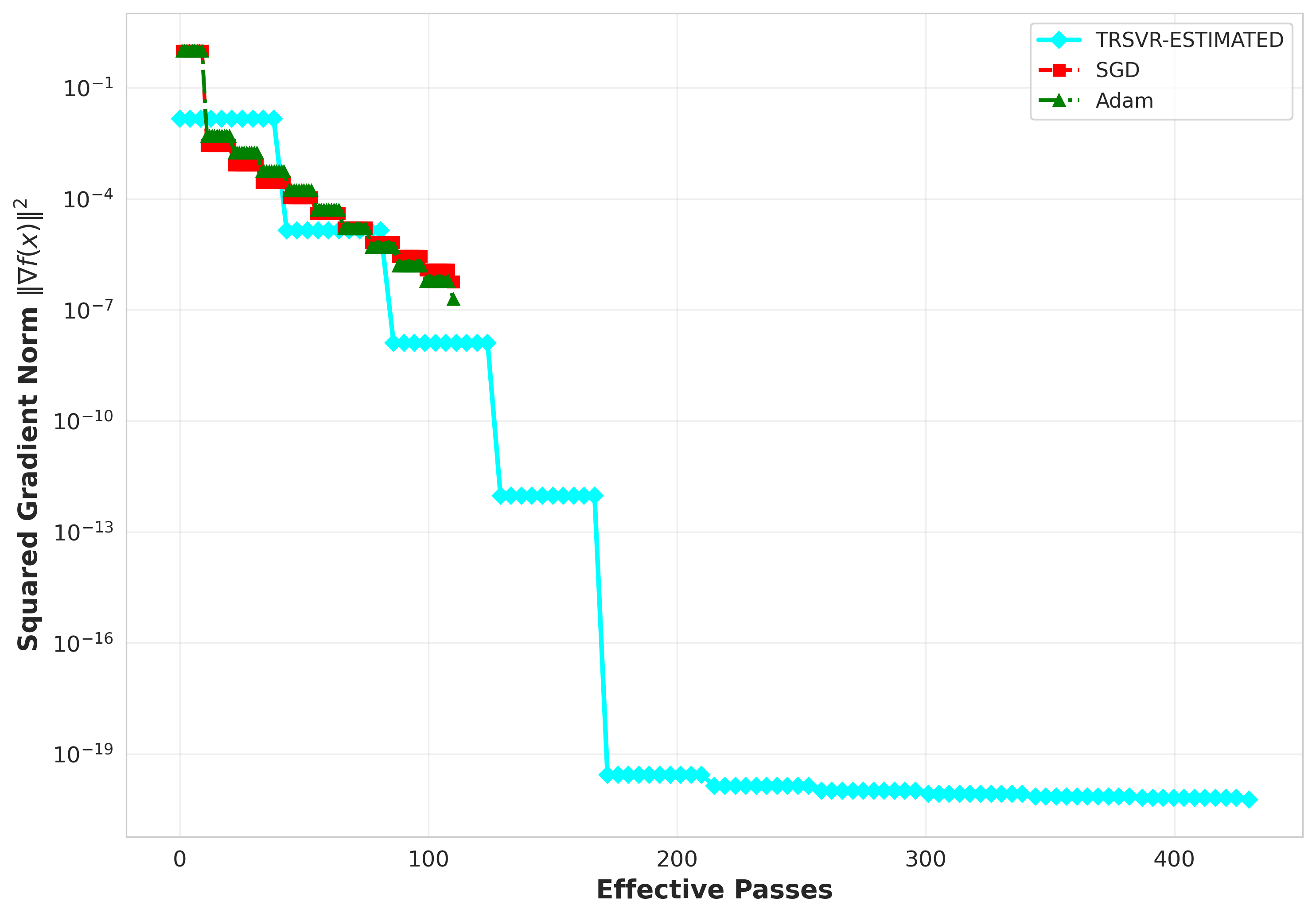}
        \caption{Squared gradient norm vs. effective passes for IJCNN1.}
        \label{fig:ijcnn1_grad_passes}
    \end{minipage}
\end{figure}

\subsection{Sensitivity analysis: inner loop length vs. batch size}

All configurations utilized the same radius-control parameter $\alpha=0.01$. Results are reported in Figures \ref{fig:rcv1_sens_gap_epochs}-\ref{fig:covertype_sens_grad_passes}.

\begin{figure}[h]
    \centering
    \begin{minipage}{0.4\textwidth}
        \centering
        \includegraphics[width=\linewidth]{ 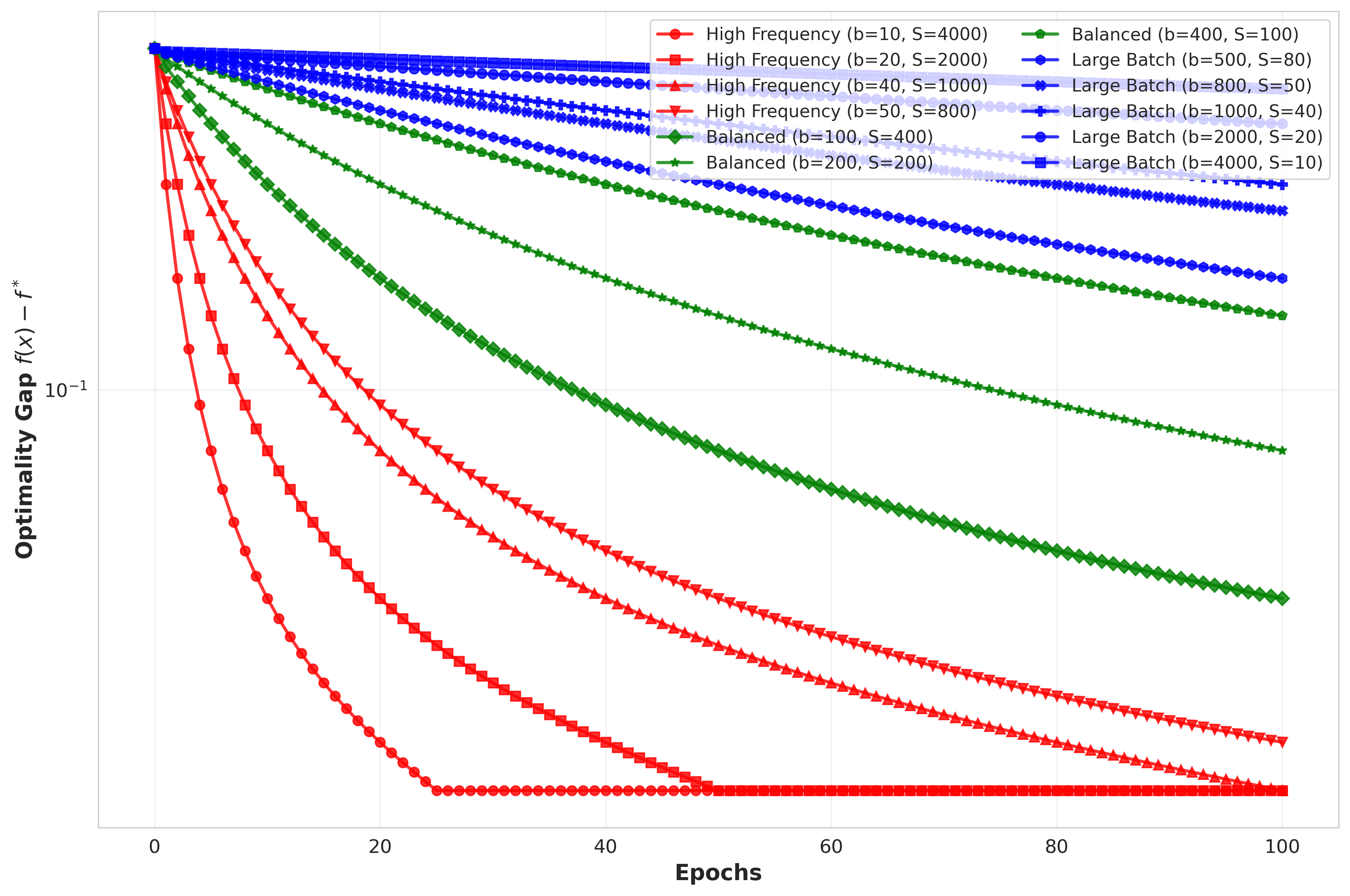}
        \caption{Optimality gap vs. epochs for RCV1.}
        \label{fig:rcv1_sens_gap_epochs}
    \end{minipage}\hfill
    \begin{minipage}{0.4\textwidth}
        \centering
        \includegraphics[width=\linewidth]{ 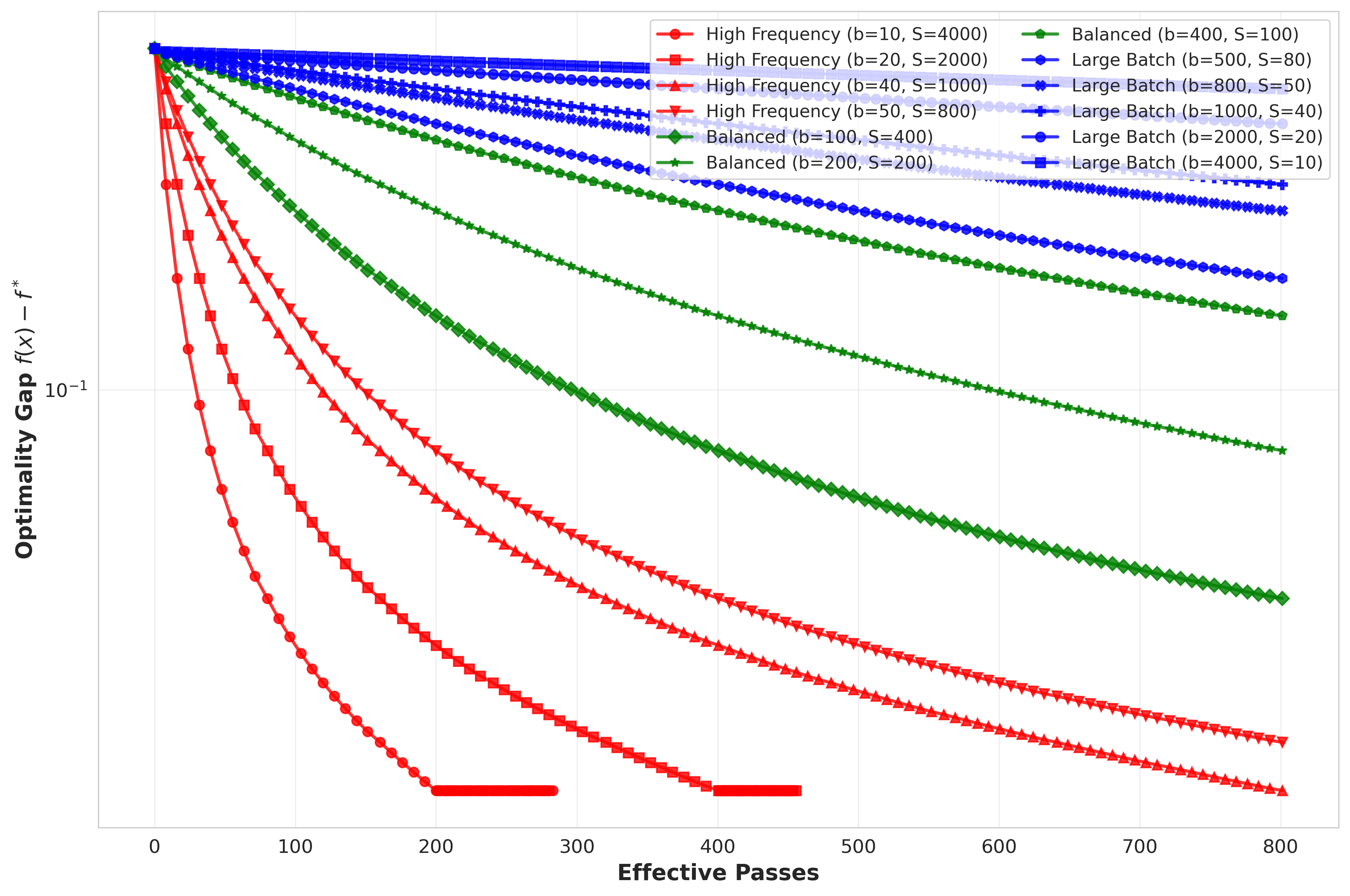}
        \caption{Optimality gap vs. effective passes for RCV1.}
        \label{fig:rcv1_sens_gap_passes}
    \end{minipage}

    \begin{minipage}{0.4\textwidth}
        \centering
        \includegraphics[width=\linewidth]{ 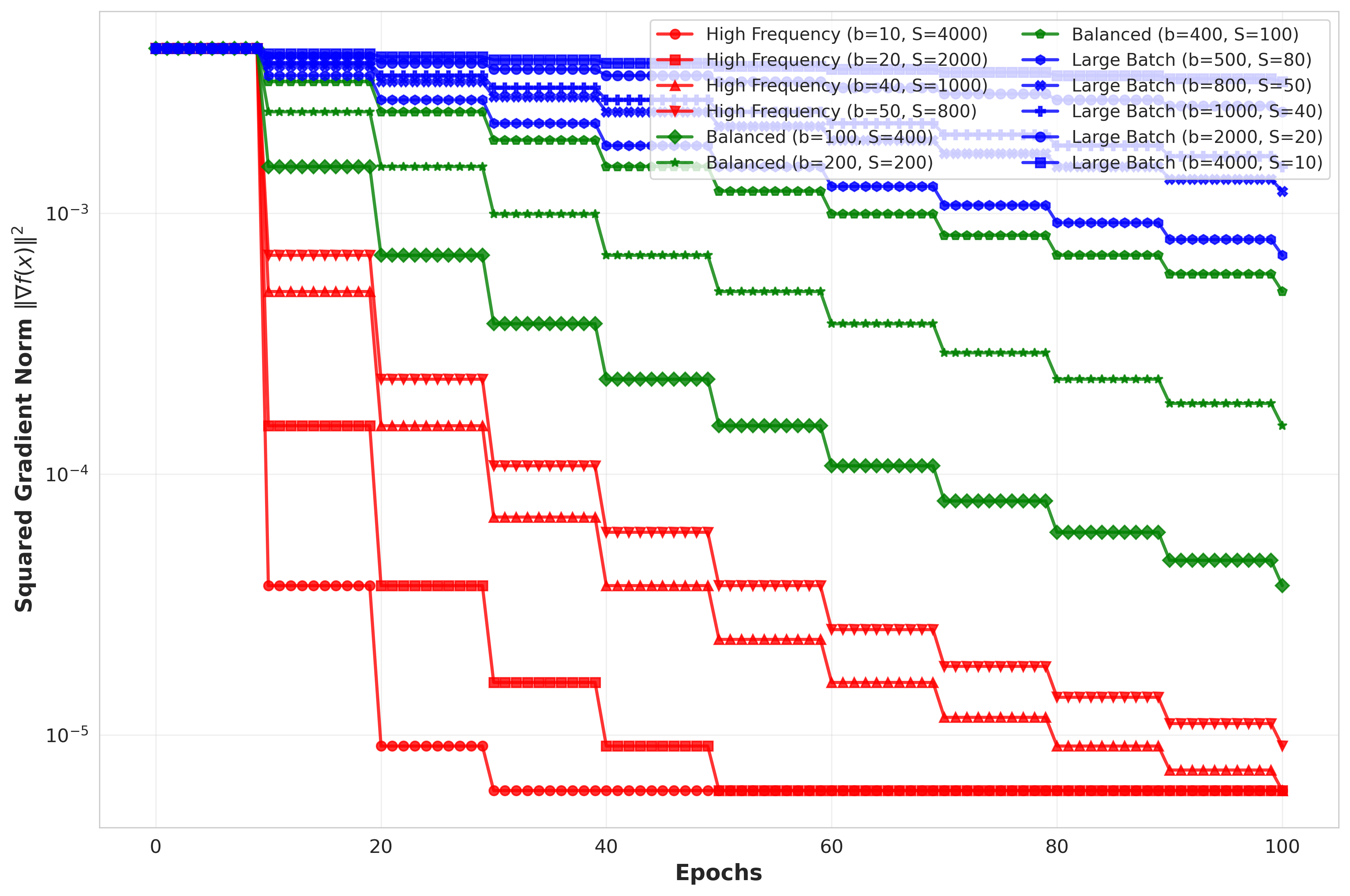}
        \caption{Squared gradient norm vs. epochs for RCV1.}
        \label{fig:rcv1_sens_grad_epochs}
    \end{minipage}\hfill
    \begin{minipage}{0.4\textwidth}
        \centering
        \includegraphics[width=\linewidth]{ Figures/Sensitivity/rcv1/rcv1_grad_norm_passes__RCV1_Sensitivity_Squared_Gradient_Norm_vs_Effective_Passes.png}
        \caption{Squared gradient norm vs. effective passes for RCV1.}
        \label{fig:rcv1_sens_grad_passes}
    \end{minipage}
\end{figure}

\begin{figure}[h]
    \centering
    \begin{minipage}{0.4\textwidth}
        \centering
        \includegraphics[width=\linewidth]{ 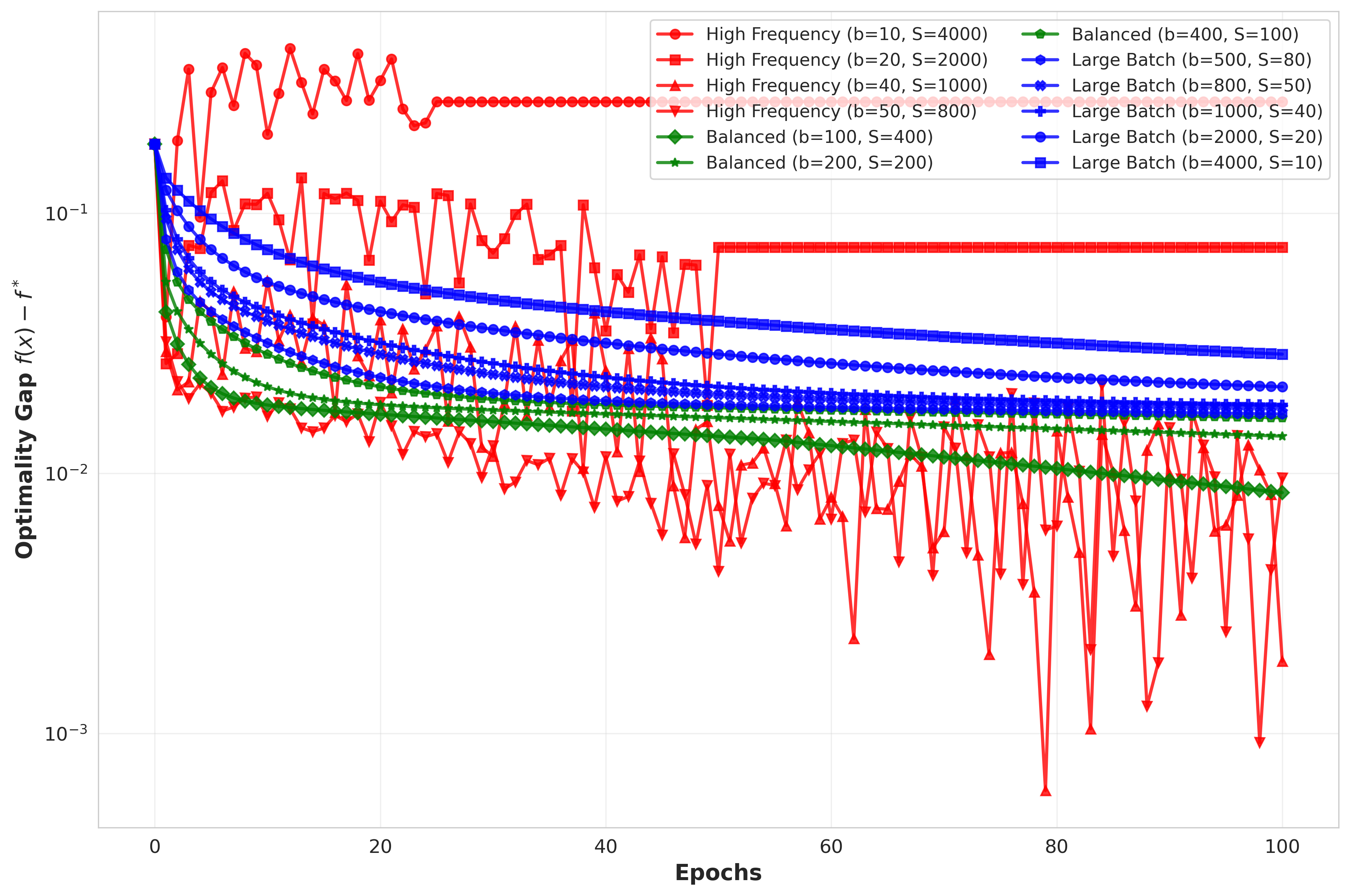}
        \caption{Optimality gap vs. epochs for Covertype.}
        \label{fig:covertype_sens_gap_epochs}
    \end{minipage}\hfill
    \begin{minipage}{0.4\textwidth}
        \centering
        \includegraphics[width=\linewidth]{ 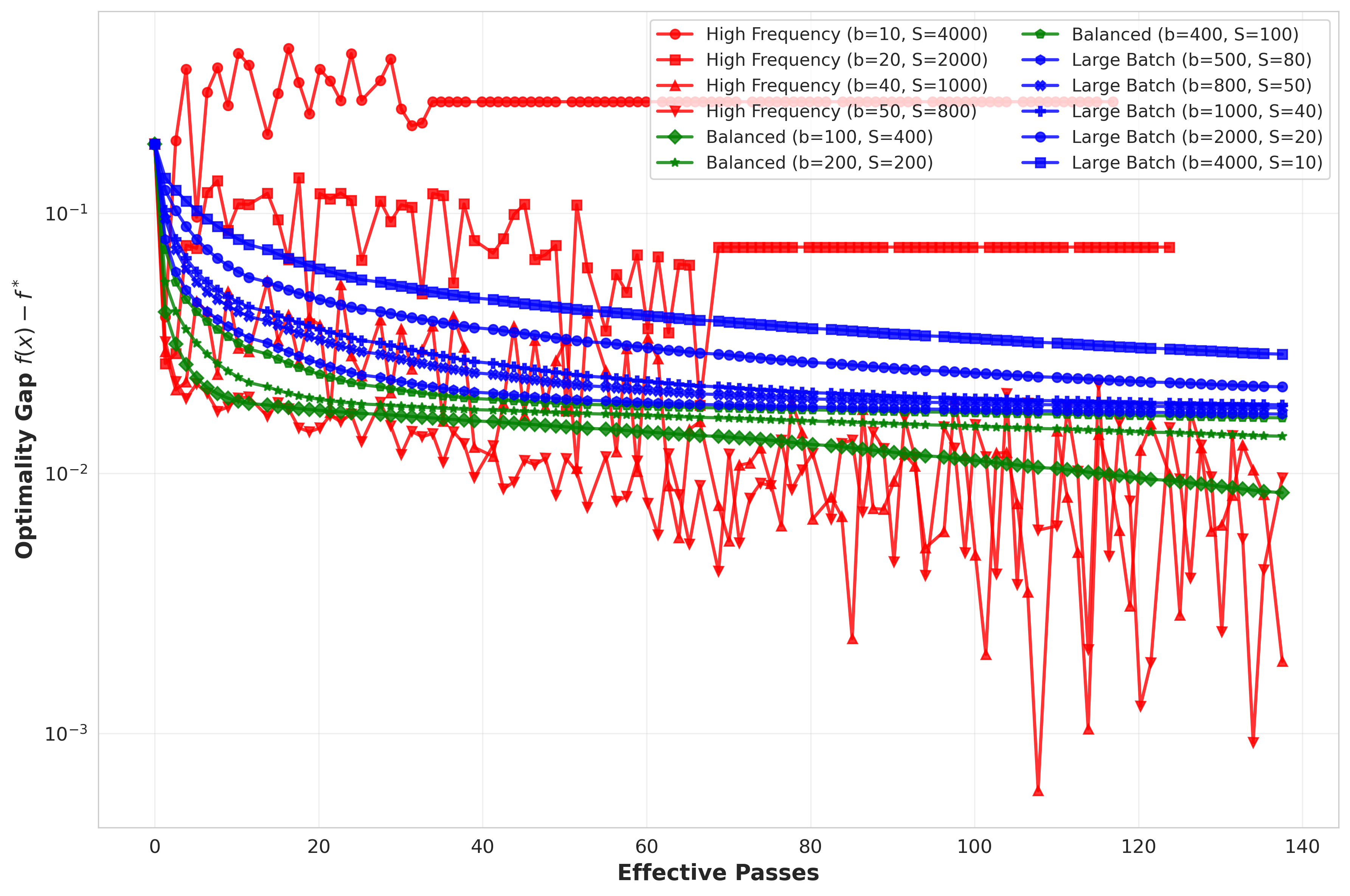}
        \caption{Optimality gap vs. effective passes for Covertype.}
        \label{fig:covertype_sens_gap_passes}
    \end{minipage}

    \begin{minipage}{0.4\textwidth}
        \centering
        \includegraphics[width=\linewidth]{ 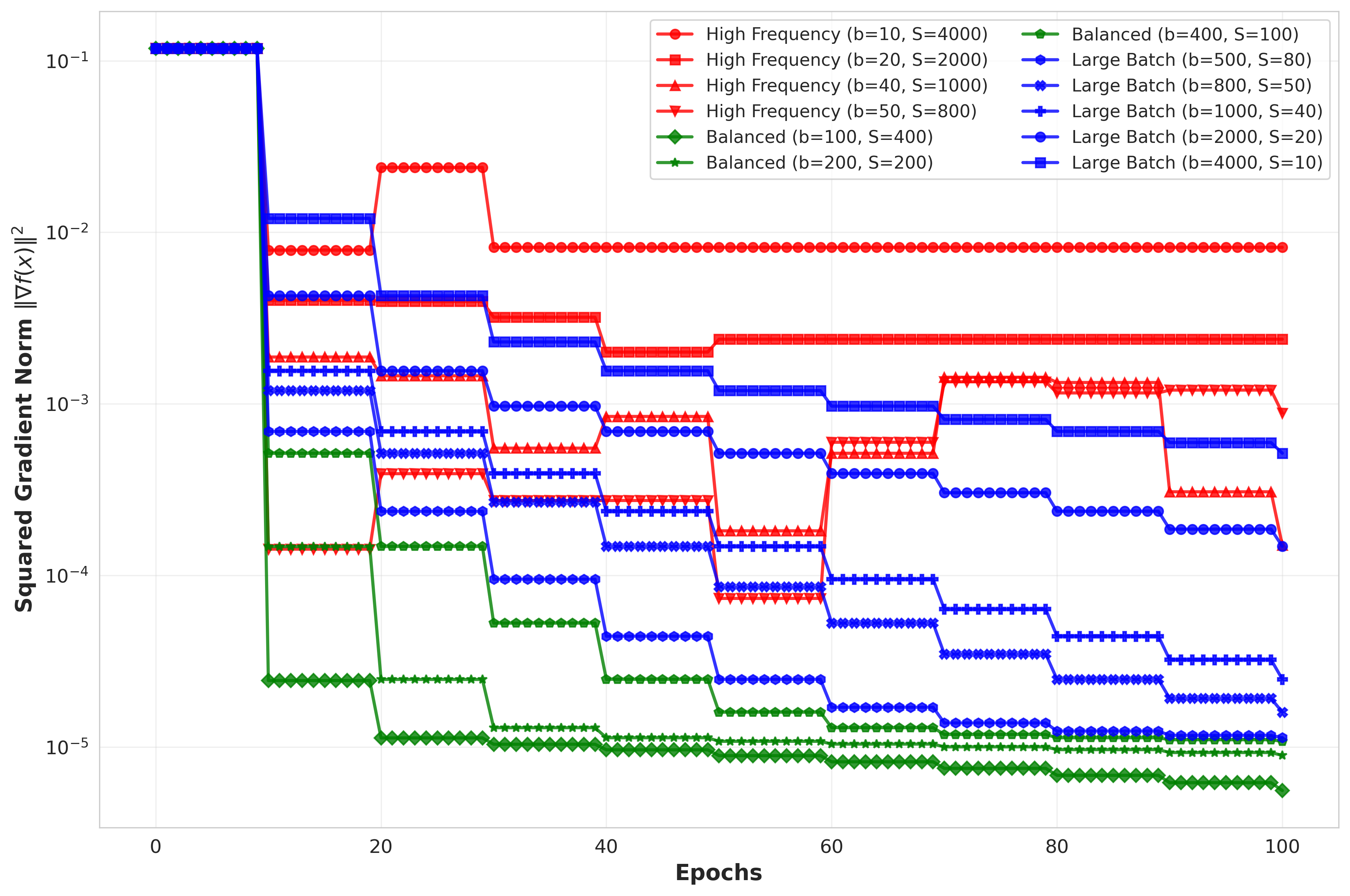}
        \caption{Squared gradient norm vs. epochs for Covertype.}
        \label{fig:covertype_sens_grad_epochs}
    \end{minipage}\hfill
    \begin{minipage}{0.4\textwidth}
        \centering
        \includegraphics[width=\linewidth]{ Figures/Sensitivity/covertype/covertype_grad_norm_passes__COVERTYPE_Sensitivity_Squared_Gradient_Norm_vs_Effective_Passes.png}
        \caption{Squared gradient norm vs. effective passes for Covertype.}
        \label{fig:covertype_sens_grad_passes}
    \end{minipage}
\end{figure}

\subsection{Ablation study}

All methods were run for a fixed budget of 100 epochs. The optimal configurations are reported in Table~\ref{tab:ablation_params}. Results are reported in Figures \ref{fig:all_datasets_ablation1}-\ref{fig:all_datasets_ablation4}.

\begin{table}[H]
    \centering
    \caption{Optimal Configurations for Ablation Study. For TRSVR, the configuration denotes $\alpha$ and inner-loop settings $(b, S)$. For Classic TR, the value is the initial trust-region radius $\Delta_0$. For TRish, the configuration denotes $\alpha$ and scaling factors $(\gamma_1, \gamma_2)$.}
    \label{tab:ablation_params}
    \begin{tabular}{lccc}
        \toprule
        \textbf{Dataset} & \textbf{TRSVR} & \textbf{Classic TR} & \textbf{TRish} \\
        \midrule
        RCV1 & $0.10$, $(20, 2000)$ & $\Delta_0 = 30.0$ & $0.40$, $(4.91, 0.03)$ \\
        Phishing & $0.09$, $(100, 400)$ & $\Delta_0 = 6.62$ & $0.007$, $(4.91, 0.49)$ \\
        Madelon & $0.08$, $(100, 400)$ & $\Delta_0 = 0.32$ & $0.14$, $(0.59, 0.008)$ \\
        Mushroom & $0.01$, $(100, 400)$ & $\Delta_0 = 14.09$ & $0.40$, $(3.22, 0.04)$ \\
        Cod-RNA & $0.26$, $(100, 400)$ & $\Delta_0 = 0.07$ & $0.14$, $(0.25, 0.002)$ \\
        IJCNN1 & $0.08$, $(20, 2000)$ & $\Delta_0 = 14.09$ & $9 \times 10^{-4}$, $(4.91, 0.13)$ \\
        \bottomrule
    \end{tabular}
\end{table}

\begin{figure*}[t]
    \centering
    \includegraphics[width=0.9\linewidth]{ 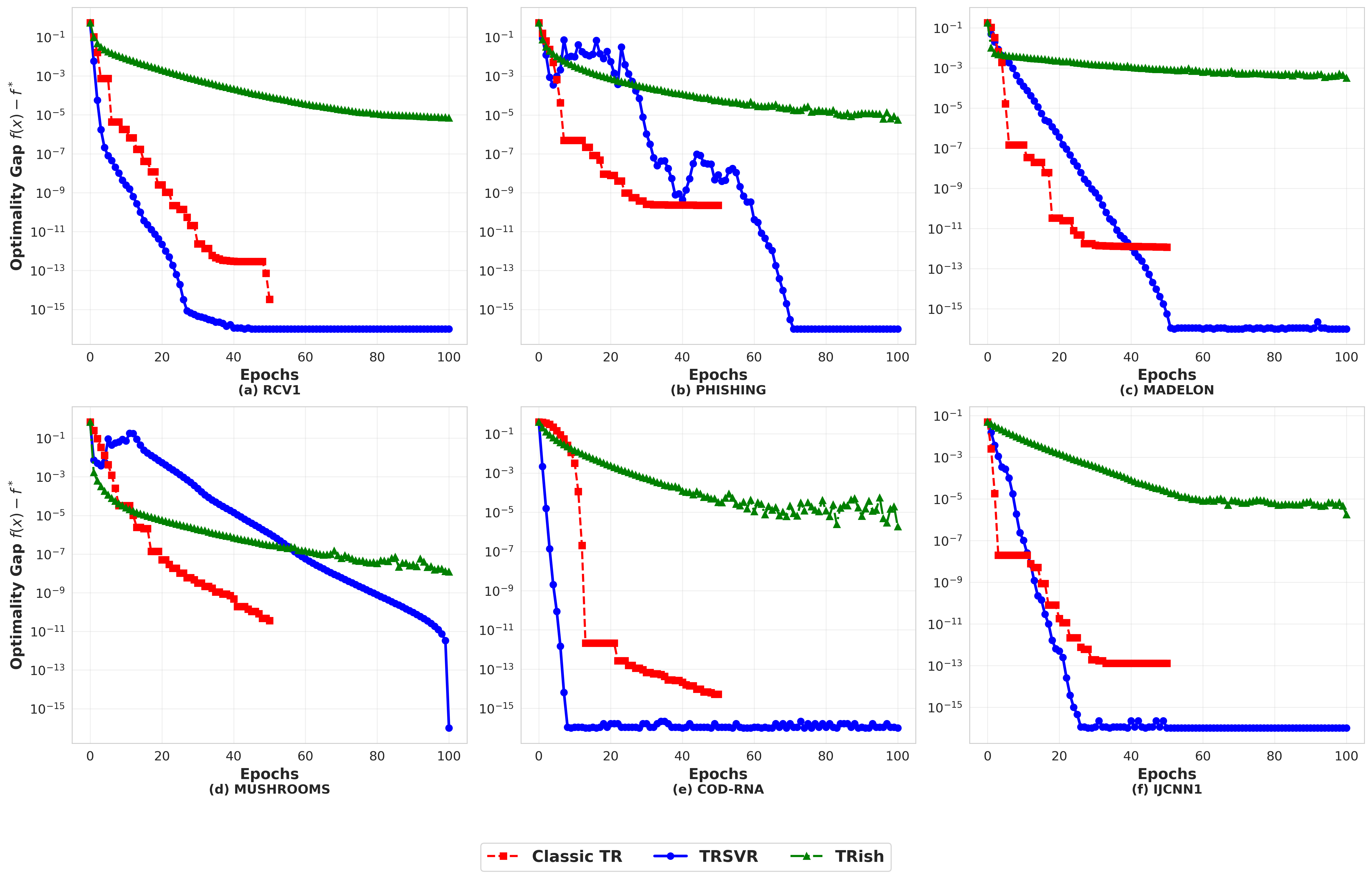}
    \caption{Optimality gap vs. epochs over six datasets.}
    \label{fig:all_datasets_ablation1}
\end{figure*}

\begin{figure*}[t]
    \centering
    \includegraphics[width=0.9\linewidth]{ 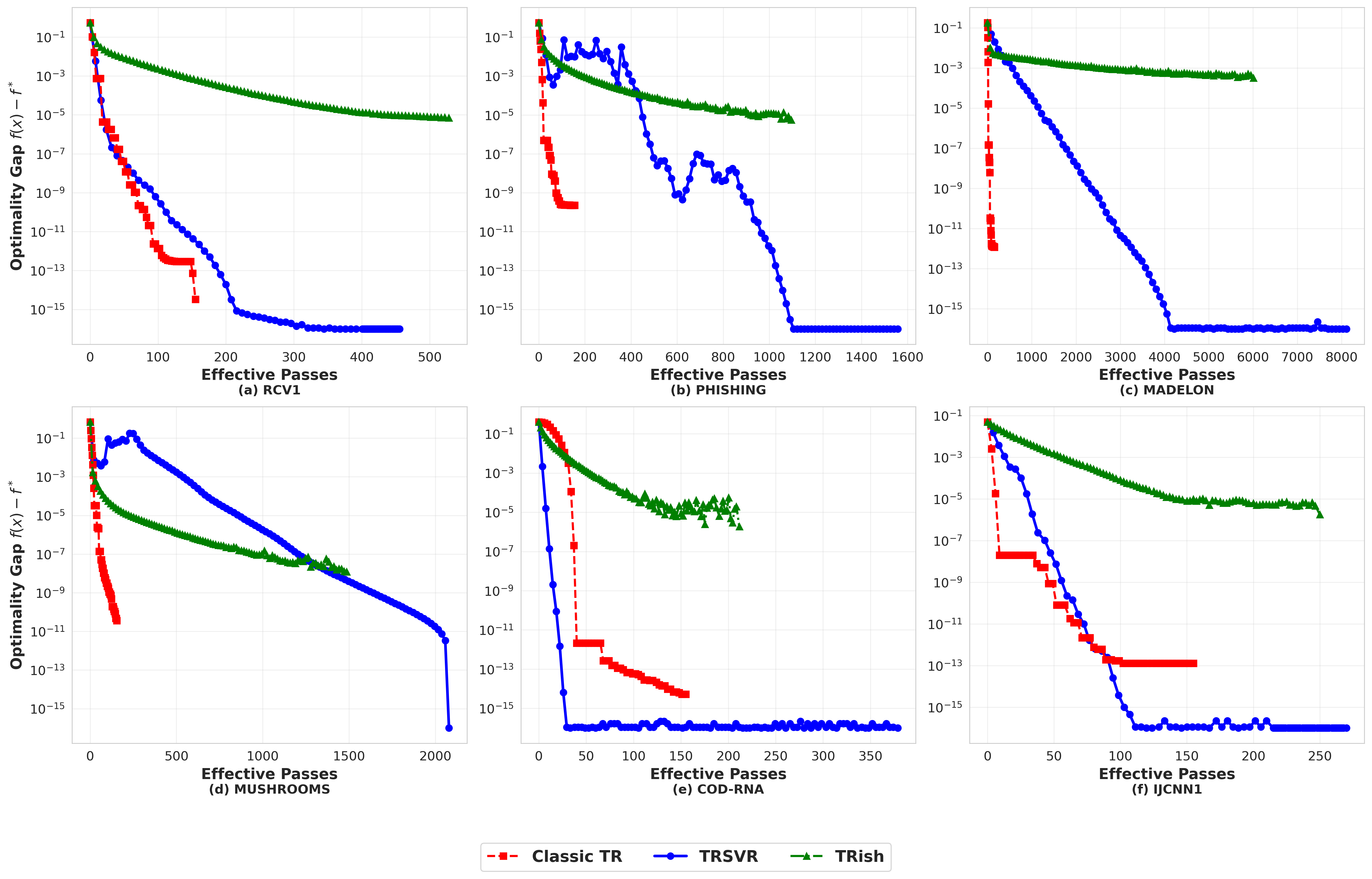}
    \caption{Optimality gap vs. effective passes over six datasets.}
    \label{fig:all_datasets_ablation2}
\end{figure*}

\begin{figure*}[t]
    \centering
    \includegraphics[width=0.9\linewidth]{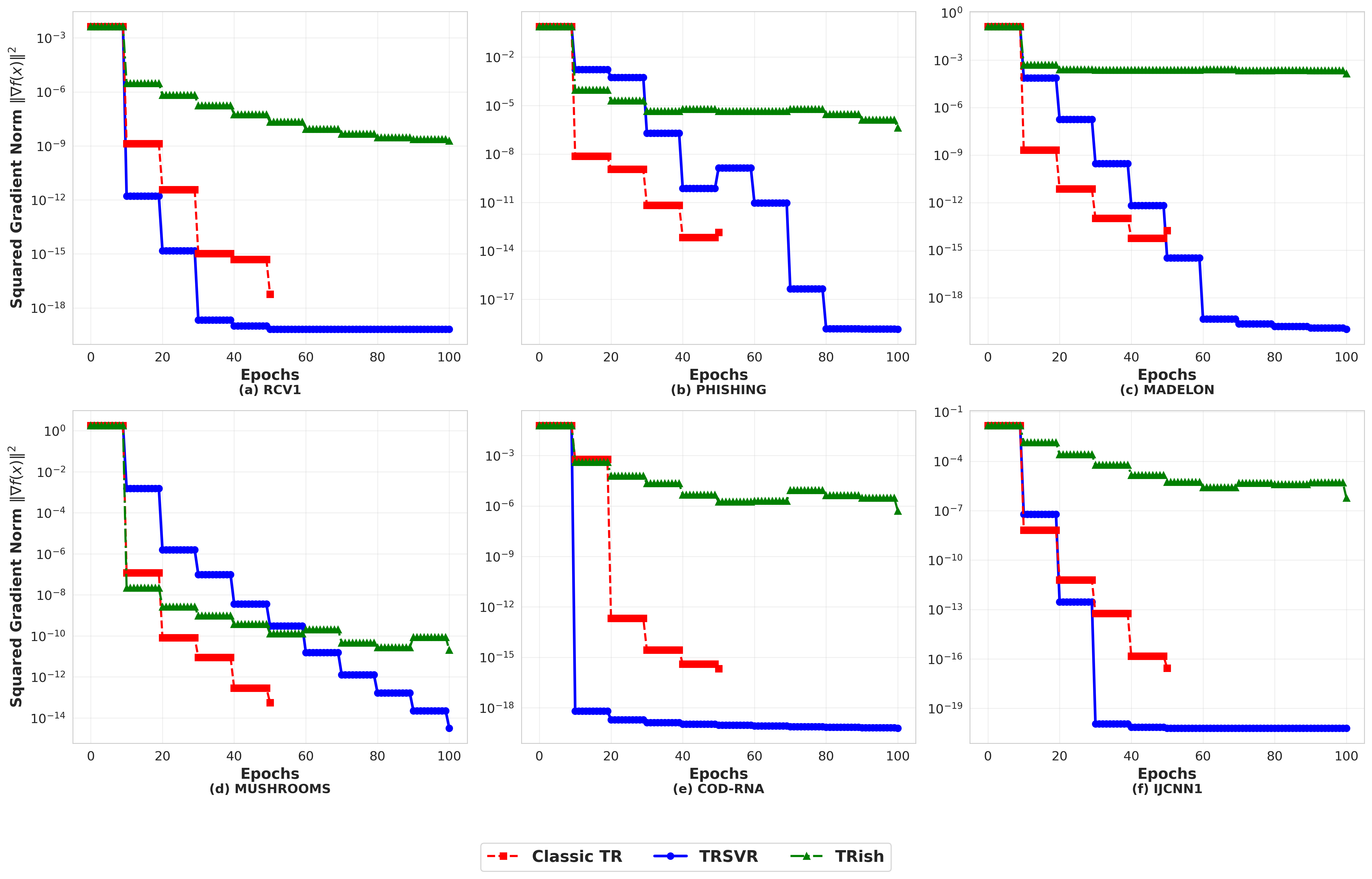}
    \caption{Squared gradient norm vs. epochs over six datasets.}
    \label{fig:all_datasets_ablation3}
\end{figure*}

\begin{figure*}[t]
    \centering
    \includegraphics[width=0.9\linewidth]{ 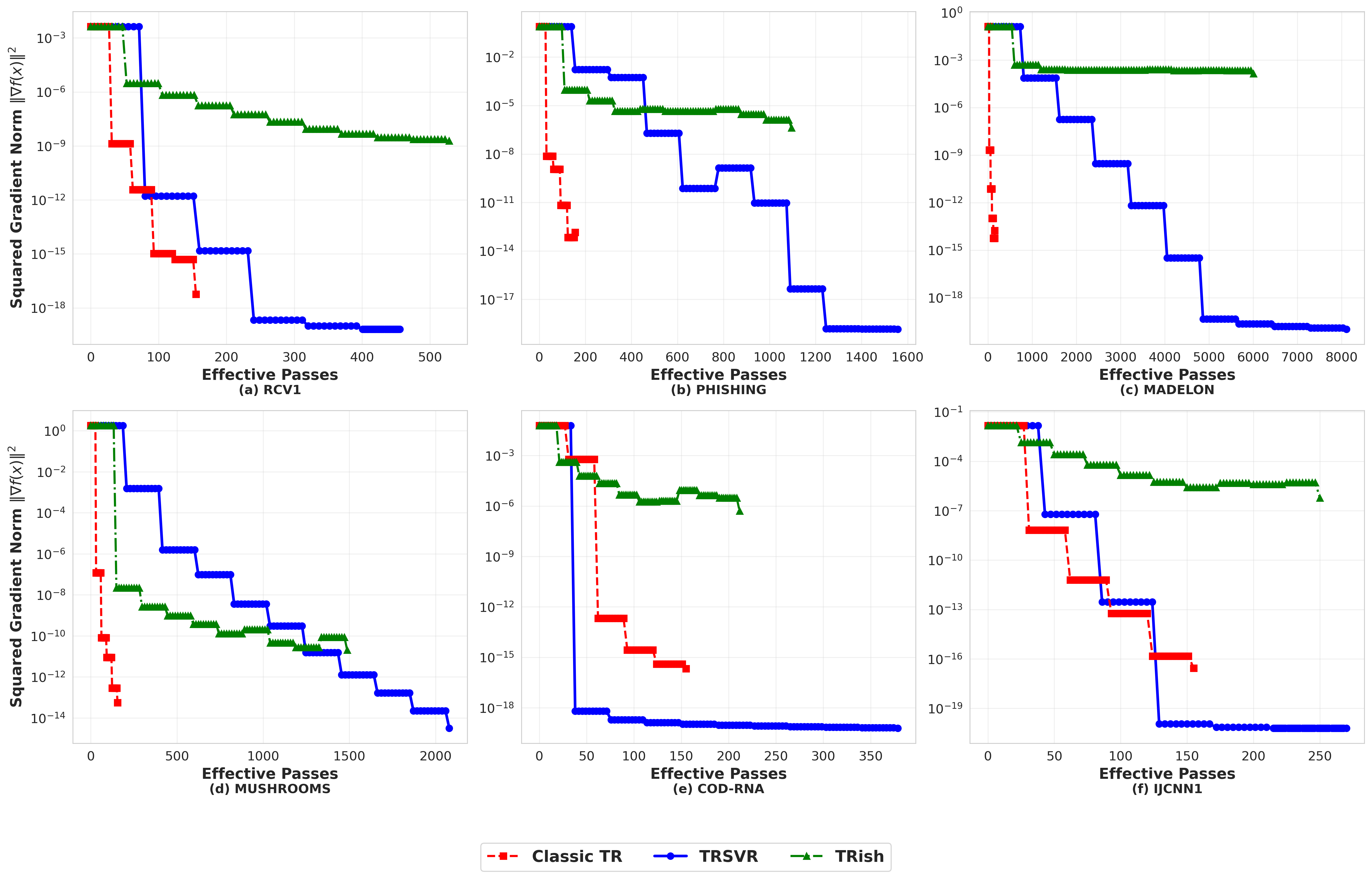}
    \caption{Squared gradient norm vs. effective passes over six datasets.}
    \label{fig:all_datasets_ablation4}
\end{figure*}

\end{document}